\documentclass[final,onefignum,onetabnum]{siamart171218}
\usepackage{braket,amsfonts}

\usepackage{array}
\usepackage{amsmath}               
  {
      \theoremstyle{plain}
      \newtheorem{assumption}{Assumption}
  }
\usepackage[caption=false]{subfig}
\usepackage{pgfplots}

\newsiamthm{claim}{Claim}
\newsiamremark{remark}{Remark}
\newsiamremark{hypothesis}{Hypothesis}
\crefname{hypothesis}{Hypothesis}{Hypotheses}

\usepackage{algorithmic}

\usepackage{graphicx,epstopdf}

\Crefname{ALC@unique}{Line}{Lines}

\usepackage{amsopn}

\usepackage{xspace}
\usepackage{bold-extra}
\usepackage[most]{tcolorbox}

\colorlet{texcscolor}{blue!50!black}
\colorlet{texemcolor}{red!70!black}
\colorlet{texpreamble}{red!70!black}
\colorlet{codebackground}{black!25!white!25}

\usepackage[normalem]{ulem}

\newcommand{\iprod}[2]{\left \langle #1, #2 \right \rangle }

\lstdefinestyle{siamlatex}{%
  style=tcblatex,
  texcsstyle=*\color{texcscolor},
  texcsstyle=[2]\color{texemcolor},
  keywordstyle=[2]\color{texemcolor},
  moretexcs={cref,Cref,maketitle,mathcal,text,headers,email,url},
}

\tcbset{%
  colframe=black!75!white!75,
  coltitle=white,
  colback=codebackground, 
  colbacklower=white, 
  fonttitle=\bfseries,
  arc=0pt,outer arc=0pt,
  top=1pt,bottom=1pt,left=1mm,right=1mm,middle=1mm,boxsep=1mm,
  leftrule=0.3mm,rightrule=0.3mm,toprule=0.3mm,bottomrule=0.3mm,
  listing options={style=siamlatex}
}

\newtcblisting[use counter=example]{example}[2][]{%
  title={Example~\thetcbcounter: #2},#1}

\newtcbinputlisting[use counter=example]{\examplefile}[3][]{%
  title={Example~\thetcbcounter: #2},listing file={#3},#1}

\DeclareTotalTCBox{\code}{ v O{} }
{ 
  fontupper=\ttfamily\color{black},
  nobeforeafter,
  tcbox raise base,
  colback=codebackground,colframe=white,
  top=0pt,bottom=0pt,left=0mm,right=0mm,
  leftrule=0pt,rightrule=0pt,toprule=0mm,bottomrule=0mm,
  boxsep=0.5mm,
  #2}{#1}

\patchcmd\newpage{\vfil}{}{}{}
	\title{A Stochastic Trust-Region Framework for Policy Optimization\thanks{Submitted to the editors DATE.
			\funding{Research supported in part by the NSFC grants 11831002 and 11421101 and by Beijing Academy  of Artificial Intelligence (BAAI).}}}
	
	\author{Mingming Zhao\thanks{Beijing International Center for Mathematical Research, Peking University, CHINA (\email{mmz102@pku.edu.cn}).}
	\and Yongfeng Li\thanks{Beijing International Center for Mathematical Research, Peking University, CHINA (\email{yongfengli@pku.edu.cn}).}
		\and Zaiwen Wen\thanks{Beijing International Center for Mathematical Research, Peking University,CHINA}(\email{wenzw@pku.eu.cn}).}
	\headers{A Stochastic Trust-Region Framework for Policy Optimization}{Mingming Zhao, Yongfeng Li,  Zaiwen Wen}

\graphicspath{ {./graph/} }

\ifpdf
\hypersetup{ pdftitle={A Stochastic Trust-Region Framework for Policy Optimization} }
\fi

\begin{document}
	\maketitle
	
\begin{abstract}
In this paper, we study a few challenging theoretical and numerical issues on the well known trust region policy optimization for deep reinforcement learning. The goal is to find a policy that maximizes the total expected reward when the agent acts according to the policy. The trust region subproblem is constructed with a surrogate function coherent to the total expected reward and a general distance constraint around the latest policy. We solve the subproblem using a preconditioned stochastic gradient method with a line search scheme to ensure that each step promotes the model function and stays in the trust region. To overcome the bias caused by sampling to the function estimations under the random settings, we add the empirical standard deviation of the total expected reward to the predicted increase in a ratio in order to update the trust region radius and decide whether the trial point is accepted. Moreover, for a Gaussian policy which is commonly used for continuous action space, the maximization with respect to the mean and covariance is performed separately to control the entropy loss. Our theoretical analysis shows that the deterministic version of the proposed algorithm tends to generate a monotonic improvement of the total expected reward and the global convergence is guaranteed under moderate assumptions. Comparisons with the state-of-the-art methods demonstrate the effectiveness and robustness of our method over robotic controls and game playings from OpenAI Gym. 
\end{abstract}


\section{Introduction}\label{sec:intro}	
In reinforcement learning, the agent starts from an initial state and interacts with the environment by executing an action from some policy iteratively. At each time step, the environment transforms the current state into the next state with respect to the action selected by the agent and gives back a reward to the agent to evaluate how good the action is, then the agent makes a new action for the next interaction based on the feedback. Repeating the above transition dynamics generates a trajectory where stores the visited states, actions and rewards. During the interactions, the transition probability and the reward function are totally determined by the environment, but the intrinsic mechanism may be mysterious. The policy characterizes the distribution of actions at each possible state. The problem is how to design a policy for the agent to maximize the total expected reward along a trajectory induced by the policy. The state-of-the-art model-free methods for reinforcement learning \cite{sutton1998reinforcement,recht2018tour} can be divided into policy-based and value-based methods. Policy-based methods directly learn or try to approximate the optimal policy by policy improvement and policy evaluation alternatively. They generate a map, i.e., a distribution from states to actions, which can be stochastic or deterministic. That is, they can be applied to both continuous and discrete action spaces. While in value-based methods, the goal is approximating the solution of the optimal Bellman equation based upon the temporal difference learning \cite{sutton1998reinforcement}. They learn a value function defined on the state-action pairs to estimate the maximal expected return of the action taken in the state. Then at each sate, the optimal policy based on the value function predicts a single action by maximizing the values.

The recent progress of deep neural networks \cite{lecun2015deep} provides many scalable and reliable learning based approaches \cite{duan2016benchmarking,haarnoja2017reinforcement,lillicrap2015continuous,schulman2015trust,schulman2017proximal} for solving large and complex real-world problems in reinforcement learning. The curse of dimensionality is conquered by expressing the value and/or policy function with a deep neural network from high-dimensional or limited sensory inputs. The deepening expedites the evolution of end-to-end reinforcement learning, also referred as deep reinforcement learning. As a representative and illuminative algorithm in deep value-based methods, deep Q-learning (DQN) \cite{mnih2015human} has succeeded in many discrete domains such as playing Atari games. The learned agent arrives at a comparable level to that of a professional human games player. They construct a Q-network to receive the raw pictures as inputs, and optimize the weights by minimizing the Bellman residual.  DQN can be viewed as a deep value iteration method directly, and some independent improvements including their combinations have been summarized in \cite{hessel2018rainbow}. The success of DQN and its variants has a restriction on the type of the problem, specifically, the maximal operator in the objective function makes the optimization to be less reliable in continuous and/or large action space. By representing the greedy action selection with a policy network, the deep deterministic policy gradient (DDPG) method \cite{lillicrap2015continuous} successfully extends the algorithmic idea of DQN into the continuous action space. The value network imitates the training in DQN and the policy network is updated by maximizing the estimated values. The two delayed deep deterministic (TD3) policy gradient algorithm \cite{fujimoto2018addressing} substantially improves DDPG by building double deep Q-networks to avoid overestimation in value estimates and delaying the policy updates to reduce the per-update error in DDPG.

Different from the optimization models based on value functions, policy-based algorithms also concentrate on optimizing the policy iteratively. In the policy improvement step, the actor updates the policy by optimizing an appropriate objective function using gradient-based methods. Policy evaluation creates a critic, i.e., a value function, to assess the policy by minimizing the Bellman error associated with the policy, which provides a basis for policy improvement. Thus the policy-based methods are usually classified as actor-critic methods. As the optimizations are practically based on the observations, the generalized advantage estimators (GAE) \cite{schulman2015high} are mostly considered for the bias-variance tradeoff and numerical stability. The discrepancy among the state-of-the-art policy-based methods mainly locates in the actor part, specifically, the surrogate function used for improving the policy. The trust region policy optimization (TRPO) \cite{schulman2015trust} generalizes the proof the policy improvement bound in \cite{kakade2002approximately} into general stochastic policies and proposes a trust region model for policy update. The model function is a local approximation of the total expected reward and the Kullback–Leibler (KL) divergence between two policies is considered as a distance constraint. The subproblem under parameterization is highly nonlinear and nonconvex rather than a typical quadratic model as in \cite{bandeira2014convergence,chen2018stochastic,wang2019stochastic} because the policy is parameterized by a neural network and the trust region constraint is replaced by a distance function of two policies. In order to develop a practical algorithm, the subproblem is approximately solved in one step by a linearization of the model function and a second-order approximation of the constraint. The proximal policy optimization (PPO) algorithm \cite{schulman2017proximal} constructs a surrogate function by modifying the model function in TRPO with a clipped probability ratio, which controls the policy update from one iteration to the next one as the KL divergence constraint does. The surrogate function is maximized by the stochastic gradient methods. Some related developments can be found in \cite{chen2018adaptive,kurutach2018model,achiam2017constrained}.

In this paper, we use a similar model as in TRPO and develop a model-free on-policy algorithm within the stochastic trust region framework for policy optimization. In TRPO,  the subproblem is solved essentially by only one natural gradient step and the size of the trust region radius is always fixed. However, in our method, the subproblem is inexactly solved by a sequence of increasing and feasible directions, and the solving process is terminated until the model function has a sufficient increase and the constraint is satisfied. Additionally, we update the trust region radius adaptively by linearly dependent on the gradient of the model and the adjustment is guided by the coherence between the surrogate function and the total expected reward.  We add the empirical standard deviation of the total expected reward to the predicted increase in a ratio to alleviate the bias caused by sampling in function estimations. This revision can be interpreted as releasing the coherence into a related confidence interval. For problems with continuous action spaces, where the Gaussian policy is commonly used, the mean of the policy tends to be sharply updated toward the best observation accompanied with an unexpected drop in variance. To avoid from the premature convergence, we separate the optimization with respect to the mean and variance into two independent parts. This alternating strategy slows down the decline of the entropy to encourage exploration which is crucial in deep reinforcement learning. In the convergence analysis, we provide a rigorous theoretical foundation for the stochastic trust region framework. With respect to the unparameterized policies, we show that the proposed trust region algorithm generates a sequence of monotonically rising policies. Moreover, specifying the total variation (TV) distance in trust region constraint, we can construct a feasible solution for the subproblem such that the model function has a sufficient improvement. We prove that the ratio can be bounded from below and it goes to one as the trust region radius goes to zero. Therefore, the total expected reward converges to the optimal value under some mild assumptions. Then we extend the results with policy parameterization. Specifically, we consider the Gaussian polices and the KL divergence in continuous action space. The alternating strategy in mean and covariance updates ensures a sufficient improvement in the model function, and the ratio is bounded below. Furthermore, under some continuous assumptions of the objective function, the mean of the policy is proved to converge to a stationary point where the covariance is also assumed to converge. The numerical performance of our method can be significantly better than that of the state-of-the-art methods including TRPO and PPO under certain typical environments.

This paper is organized as follows. In Section \ref{Preliminaries}, we introduce basic preliminaries in deep reinforcement learning, and review some related algorithms. The trust region framework for policy optimization is proposed in Section \ref{Our Method}. In Section \ref{Theoretical Analysis}, we give a convergence analysis of our algorithmic framework for problems with parameterized and unparameterized policies, respectively. The alternating strategy for the Gaussian policies is included in the proof. To handle practical problems, we present a stochastic trust region algorithm for policy optimization named STRO in Section \ref{Empirical Algorithm}. Finally, extensive numerical experiments are presented in Section \ref{Experiments} to demonstrate the effectiveness and robustness of our algorithm in MuJoCo \cite{todorov2012mujoco} and Atari games \cite{bellemare2013arcade}.

\section{Preliminaries}\label{Preliminaries}
\subsection{Notations}\label{notation}
 We consider the reinforcement learning problem defined by an infinite-horizon discounted Markov decision process (MDP). The MDP is denoted as a tuple $(\mathcal{S},\mathcal{A},P,r,\rho_0, \gamma)$, where $\mathcal{S}$ and $\mathcal{A}$ are known as the state and action spaces,  respectively. $P: \mathcal{S}\times \mathcal{A}\times \mathcal{S}\rightarrow \mathbf{R}$ is the transition probability distribution, $r: \mathcal{S} \times \mathcal{A} \rightarrow \mathbf{R}$ is the bounded reward function, $\rho_0: \mathcal{S}\rightarrow \mathbf{R}$ is the distribution of the initial state $s_0$, and $\gamma \in (0,1)$ is the discount factor.

In this paper, we consider the stochastic policy, which maps the state into a distribution over action space:
\begin{equation*}
\pi(a|s):\ \mathcal{S}\times\mathcal{A}\rightarrow[0,1],\ \sum_{a\in\mathcal{A}}\pi(a|s)=1,\forall s\in\mathcal{S}.
\end{equation*}
The total expected reward is the expectation of the cumulative discounted reward of a trajectory induced by the policy $\pi$:
\begin{equation}\label{etapi-form}
\eta(\pi)=\mathbb{E}_{\pi}\left[\sum_{t=0}^\infty\gamma^tr(s_t, a_t)\right],
\end{equation}
where $s_0 \sim \rho_0, a_t \sim \pi(\cdot|s_t), s_{t+1}\sim P(\cdot|s_t,a_t)$. Moreover, we introduce the unnormalized discounted visitation frequency with respect to $\pi$ as:
\begin{equation*}
\rho_{\pi}(s)=\sum_{t=0}^{\infty}\gamma^{t}\mathbb{P}(s_{t}=s|\pi).
\end{equation*}
We take the following standard definitions of the action value function $Q_\pi$ and the state value function $V_\pi$:
		\begin{equation*}
		\begin{split}
	Q_\pi(s,a) &= \mathbb{E}_{\pi}\left[\sum_{l=0}^\infty\gamma^lr(s_{t+l}, a_{t+l})\Big|s_0=s,a_0=a\right], \\
	V_\pi(s) &= \mathbb{E}_{\pi}\left[\sum_{l=0}^\infty\gamma^lr(s_{t+l}, a_{t+l})\Big|s_0=s\right].
	\end{split}
	\end{equation*}
 They satisfy the recursive relationships:
     \begin{equation*}
     Q_\pi(s,a) =  r(s,a)+\gamma\mathbb{E}_{P(\cdot|s,a)}\left[V_{\pi}(s')\right],\ V_{\pi}(s) = E_{\pi(\cdot|s)}\left[Q_{\pi}(s,a)\right].
     \end{equation*}
The advantage of an action $a$ over a given state $s$ is defined as 
\begin{equation*}
A_\pi(s,a) = Q_\pi(s,a)-V_\pi(s).
\end{equation*}

\subsection{The optimization model}	
Generally, we consider the optimization
\begin{equation}\label{etapi}
\max_{\pi\in\Pi}\quad\eta(\pi),
\end{equation}
where $\Pi=\{\pi|\sum\limits_{a\in\mathcal{A}}\pi(a|s)=1,\pi(a|s)\geq0,\forall a\in\mathcal{A}, s\in\mathcal{S}\}$. For any fixed policy $\pi$, the total expected reward of another policy $\tilde{\pi}$ can be expressed in terms of the expectation of advantage function $A_{\pi}(s,a)$ over policy $\tilde{\pi}$:
\begin{equation}
\eta(\tilde{\pi})=\eta(\pi)+\sum_{s}\rho_{\tilde{\pi}}(s)\sum_{a}\tilde{\pi}(a|s)A_{\pi}(s,a).
\end{equation}
The derivation can be referred in \cite{kakade2002approximately,schulman2015trust}. The complex dependency of $\rho_{\tilde{\pi}}$ on $\tilde{\pi}$ motivates a popular surrogate function \cite{schulman2015trust} by approximating $\rho_{\tilde{\pi}}$ with $\rho_{\pi}$:
\begin{equation}\label{L}
\begin{split}
L_{\pi}(\tilde{\pi})&=\eta(\pi)+\sum_{s}\rho_{\pi}(s)\sum_{a}\tilde{\pi}(a|s)A_{\pi}(s,a)\\
&=\eta(\pi)+\mathbb{E}_{\rho_{\pi},\pi}\left[\frac{\tilde{\pi}(a|s)}{\pi(a|s)}A_{\pi}(s,a)\right].
\end{split}
\end{equation}
The second equality follows from importance sampling, where the sampling distribution is $\pi$. Apparently, $L_{\pi}$ is a linear function in the policy space and intersects $\eta$ at $\pi$. The maximizer of $L_{\pi}$ over $\tilde{\pi}$ is exactly the greedy policy update at $\pi$ in policy iteration, since for each state $s$, the action with the maximal advantage value is assigned with probability one:
\begin{equation*}
\tilde{\pi}^{*}=\mathop{\arg\max}\limits_{\tilde{\pi}\in\Pi}L_{\pi}(\tilde{\pi}),\ \tilde{\pi}^{*}(a|s)=\begin{cases}1,&\mbox{}\  a=\mathop{\arg\max}\limits_{a'}A_{\pi}(s,a'),\\
0,&\mbox{}\  \mathrm{otherwise.}\end{cases}
\end{equation*}
Besides, the evaluation of $L_{\pi}$ at any policy $\tilde{\pi}$ only requires the expectation over $\pi$, which is much cheaper than that of $\eta$.

In deep reinforcement learning, to address the high dimensionality and complexity in real-world tasks, the policy $\pi$ is usually parameterized by a set of variables, for example, a differentiable neural network weighted by $\theta$. The constraint on the policy, i.e., $\pi\in\Pi$, is guaranteed by the parameterization. For simplicity, we can overload our previous notations to associate with $\theta$ rather than $\pi$, e.g., $\eta(\theta):=\eta(\pi_{\theta})$, $Q_\theta(s_t,a_t):=Q_{\pi_\theta}(s_t,a_t)$, $L_{\theta}(\tilde{\theta}):=L_{\pi}(\tilde{\pi})$ and $\rho_{\theta}(s):=\rho_{\pi_{\theta}}(s)$. The goal is maximizing the total expected reward \eqref{etapi-form} in the parameterized policy space:
\begin{equation}\label{eq: origionalobjective}
\max_{\theta}\quad\eta(\theta).
\end{equation}
The optimal solution $\theta^{*}$ corresponds to the so-called optimal policy $\pi_{\theta^{*}}$, which equivalently implies that for any $s\in\mathcal{S},\ a\in\mathcal{A}$:
\begin{equation*}
Q_{\theta^{*}}(s,a)\geq Q_{\theta}(s,a),\ V_{\theta^{*}}(s)\geq V_{\theta}(s),\mathrm{\forall\theta}.
\end{equation*} 
In general, for any MDP with a differentiable policy $\pi_{\theta}$, the policy gradient \cite{sutton2000policy} is formulated as
	\begin{equation}\label{gradient}
	\nabla \eta(\theta) = \sum_{s}\rho_{\theta}(s)\sum_{a} \nabla\pi_\theta(a|s)A_{\theta}(s,a)=\mathbb{E}_{\rho_{\theta},\pi_{\theta}}\bigg[ \nabla \log\pi_\theta(a|s)A_{\theta}(s,a) \bigg].
	\end{equation}
Intuitively, $L_{\theta}(\tilde{\theta})$ matches $\eta(\tilde{\theta})$ up to the first-order accuracy at $\theta$ in the parameterized policy space:
\begin{equation}\label{1stmatching}
L_{\theta}(\theta)=\eta(\theta),\ \nabla L_{\theta}(\theta) = \nabla\eta(\theta).
\end{equation}

\subsection{Related algorithms}
Policy gradient type methods \cite{kakade2002natural,sutton1998reinforcement,sutton2000policy} directly take a stochastic gradient version of \eqref{gradient} to update the policy in an incremental manner:
\begin{equation}\label{pgupdate}
\theta_{k+1} =\theta_k +\alpha M(\theta_k)\nabla\eta(\theta_k),
\end{equation}
where $\alpha>0$ is the step size, and $M(\theta_k)$ is a preconditioning matrix that may associate with $\theta_k$. The policy gradient type algorithms distinguish from each other with the choice of the preconditioning matrix $M(\theta_k)$ \cite{furmston2016approximate}. The vanilla policy gradient (VPG) \cite{sutton2000policy} method using $M(\theta_k)=I$ often suffers poor-scaled issues. The natural policy gradient (NPG) algorithm \cite{kakade2002natural} brings the natural gradient techniques originated from the neural networks and takes $M(\theta_k)$ as the inverse of the Fisher information matrix (FIM) of the policy $\pi_{\theta_k}$, that is,
\begin{equation}\label{}
M(\theta_k)^{-1} = \mathbb{E}_{\rho_{\theta_k},\pi_{\theta_k}}\left[\nabla\log\pi_{\theta_k}(s,a)\nabla\log\pi_{\theta_k}(s,a)^T\right].
\end{equation}
It can be showed as defining a matrix norm on the parameterized policy space. Generally, a well-tuned update rule for the step size $\alpha$ in \eqref{pgupdate} is crucial for the numerical stability in both VPG and NPG.
%

In TRPO, they propose to optimize a surrogate function coherent to the total expected reward with a constraint on the KL divergence of two policies, which controls how far the policy is allowed to update. At each iteration, they construct a trust region subproblem
\begin{equation}
\max_{\theta} \quad  L_{\theta_{k}}(\theta), \quad \mathrm{s.t.}\ \mathbb{E}_{s \sim \rho_{\theta_{k}}} [D_{KL}(\pi_{\theta_{k}}(\cdot |s)||\pi_{\theta}(\cdot |s))]\leq \delta,
\label{eq: trpoobjectivewithconst}
\end{equation}
where $\delta$ is a fixed constant and the KL divergence $D_{KL}(p||q)=\sum\limits_xp(x)\log\frac{p(x)}{q(s)}$. To develop a practical algorithm for solving the subproblem, they take a linear approximation of the model function and a second-order approximation of the constraint. Essentially, TRPO solves the quadratic constrained optimization:
\begin{equation}\label{trpoactual}
\max_{\theta} \quad  \nabla L_{\theta_{k}}(\theta_k)^{T}(\theta-\theta_k), \quad \mathrm{s.t.}\ \frac{1}{2}(\theta-\theta_k)^TH(\theta_k)(\theta-\theta_k)\leq  \delta,
\end{equation}
where $H(\theta_k)$ is the FIM of the policy $\pi_{\theta_k}$ as well as the second-order approximation of the KL divergence at $\theta_k$. Then the solution of \eqref{trpoactual} obtained by the conjugate gradient method \cite{Nocedal1999Numerical,sun2006optimization} is taken as an approximate solution of \eqref{eq: trpoobjectivewithconst}. Obviously, the induced update direction is collinear with that of NPG, therefore TRPO can be viewed as a natural policy gradient algorithm with self-adaptive step size.

PPO optimizes the model function in \eqref{eq: trpoobjectivewithconst} with point-wise value clipping to penalize a large policy update. They construct a clipped surrogate function at each iteration:
\begin{equation*}
L^{clip}_{k}(\theta) = \mathrm{E}_{\rho_{\theta_k},\pi_{\theta_{k}}}\left[\min (r_{k}(s,a)A_{\theta_{k}}(s,a),\mathrm{clip}(r_{k}(s,a),1-\epsilon, 1+\epsilon)A_{\theta_{k}}(s,a))    \right],
\end{equation*}
where $r_{k}(s,a) = \frac{\pi_{\theta}(a|s)}{\pi_{\theta_{k}}(a|s)}$ is the probability ratio, $\mathrm{clip}(x,r_1,r_2)=\min(\max(x,r_1),r_2)$, and $\epsilon$ is a hyper-parameter. It is also a local estimate of the total expected reward, and has stringent control of the policy update. The maximization of $L^{clip}_{k}(\theta)$ is approximately solved by multiple epochs of the stochastic gradient methods.
\label{eq: trpoobjectivewithconstr}

\section{A Trust Region Method for Policy Optimization}
As we showed in the last section, the function $L_{\theta_k}(\theta)$ gives a good local approximation to $\eta(\theta)$ around $\pi_{\theta_k}$ as in \eqref{1stmatching} and it has much lower costs than $\eta$ in evaluation and derivation. These advantages motivate an underlying exploration of $L_{\theta_k}$ to extract potential information for policy improvement. At the $k$-th iteration, we construct the trust region model as:
\begin{equation}\label{trposub}
\max_{\theta} \quad L_{\theta_{k}}(\theta), \quad \mathrm{s.t.}\ \mathbb{E}_{s \sim \rho_{\theta_{k}}}\left[D(\pi_{\theta_{k}}(\cdot|s),\pi_{\theta}(\cdot|s)) \right]\leq \delta_k,
\end{equation}
where $D$ is a general metric function of two distributions, such as the KL divergence and TV distance. The subproblem \eqref{eq: trpoobjectivewithconst} in TRPO can be viewed as a special case of ours by taking the KL divergence as the distance metric function. We embed the subproblem \eqref{trposub} into a general trust region method, so as to monitor the acceptance of the trial point and to adjust the trust region radius $\delta_k$. 

\subsection{Algorithmic Framework}\label{Our Method}

We now state the trust region framework for policy improvement with a random initialization. At the $k$-th iteration, the algorithm approximately solves the subproblem \eqref{trposub} to obtain a trial point $\tilde{\theta}_{k+1}$, then we resolve two issues: 1) whether to accept $\tilde{\theta}_{k+1}$ as $\theta_{k+1}$; 2) how to update the trust region radius $\delta_{k}$ adaptively. The canonical trust region framework leads us to compute a ratio to evaluate the agreement between the objective function and the surrogate function at $\tilde{\theta}_{k+1}$:
 \begin{equation}\label{ratio}
r_{k}= \frac{\eta(\tilde{\theta}_{k+1})-\eta(\theta_{k})}{L_{\theta_k}(\tilde{\theta}_{k+1})-L_{\theta_k}(\theta_{k})}.
\end{equation}
The trial point $\tilde{\theta}_{k+1}$ is accepted if $r_{k}$ is greater than some positive constant $\beta_0$ and it is called a successful iteration, i.e., $\theta_{k+1}=\tilde{\theta}_{k+1}$, otherwise, the iteration is unsuccessful and $\theta_{k+1}=\theta_{k}$:
\begin{equation}\label{updatetheta}
\theta_{k+1}=\begin{cases}\tilde{\theta}_{k+1},&\mbox{}\  r_{k}\geq\beta_0,\\
\theta_{k},&\mbox{}\ \mathrm{otherwise.}
\end{cases}
\end{equation}
The adjustment of the trust region radius $\delta_{k}$ is based on the ratio as:
\begin{equation}\label{updateC}
\delta_{k+1}=\begin{cases}\gamma_1\delta_{k},&\mbox{}\  r_{k}\geq\beta_1,\\
\gamma_2\delta_{k},&\mbox{}\  r_{k}\in[\beta_0,\beta_1),\\
\gamma_3\delta_{k},&\mbox{}\ \mathrm{otherwise,}
\end{cases}
\end{equation}
where $0<\beta_0<\beta_1$, and $0<\gamma_3<\gamma_2\leq1<\gamma_1$. These tuning parameters control the accuracy of the model by determining how aggressively the trust region radius is updated when an iteration is successful or not. The training process is summarized in Algorithm \ref{algo-expectation-form}.

\begin{algorithm}[tbp]
		\caption{A Trust Region Optimization Framework}
		\label{algo-expectation-form}
		\begin{algorithmic}[1]
			\REQUIRE Set $\theta_{0}$, $\delta_0$, $k=0$
			\WHILE{$\mathrm{stopping\ criterion\ not\ met}$}	
					
			\STATE solve \eqref{trposub} to obtain a trail point $\tilde{\theta}_{k+1}$;
			
		\STATE compute the ratio $r_{k}$ via \eqref{ratio};
	\STATE update $\theta_{k+1}$ using \eqref{updatetheta};		
			\STATE update $\delta_{k+1}$ using \eqref{updateC};
			\STATE $k=k+1$; 
						\ENDWHILE
\end{algorithmic}
	\end{algorithm}

\subsection{Theoretical Analysis}\label{Theoretical Analysis}
We next establish the convergence of the trust region framework for policy optimization. The analysis starts from the unparameterized case, then moves to the case that the policy is parameterized as Gaussian distributions. During the discussion, the states set $\mathcal{S}$ is supposed to be finite and the initial state distribution $\rho_0$ is assumed that 
\begin{equation*}	
\rho_{0}(s)>0,\forall s\in \mathcal{S}.
\end{equation*}
Consequently, for any policy $\pi$, it holds
 \begin{equation}\label{rhocond}
\rho_{\pi}(s) \geq \rho_0(s)  > 0,\forall s\in \mathcal{S}.
\end{equation} 
\subsubsection{Unparameterized Policy}
 In this part, we consider a MDP with a finite set of actions $\mathcal{A}$, and focus on the policy $\pi$ itself, i.e., considering the case that $\theta = \{\pi(a|s): \forall s \in \mathcal{S},a \in \mathcal{A}\}$, which means $\pi_{\theta}=\pi$. Note that our analysis can be similarly extended to continuous action spaces (Eucliean space). Typically, we consider the following subproblem
\begin{equation}\label{trposubanaly}
\max_{\pi}\quad L_{\pi_k}(\pi),\ \mathrm{s.t.\ } \mathbb{E}_{s \sim \rho_{\pi_{k}}} [D_{TV}(\pi_{k}(\cdot |s)||\pi(\cdot |s))] \leq \delta_k,
\end{equation}
where $\mathrm{D}_{TV}(p||q)=\frac{1}{2}\sum\limits_{x}|p(x)-q(x)|$ is the total variation distance between two distributions $p$ and $q$. Usually, a close-form solution of \eqref{trposubanaly} is unknown.
We next define the so-called policy advantage to introduce the theoretical results that follow. 
\begin{definition}[Policy Advantage]
The policy advantage $\mathbb{A}_{\pi}(\pi')$ of a policy $\pi'$ with respect to a policy $\pi$ is defined by
	\begin{equation*}
	\mathbb{A}_{\pi}(\pi') = \mathrm{E}_{s\sim\rho_{\pi}}\left[\mathrm{E}_{a\sim\pi'(\cdot|s)}\left[A_{\pi}(s,a)\right]\right].
	\end{equation*}
\end{definition}
From the definition, it is straightforward to obtain $L_{\pi_k}(\pi)=\eta(\pi_k)+\mathbb{A}_{\pi_{k}}(\pi)$. 
In the next lemma, we give an optimality condition of the RL problems \eqref{etapi}.
\begin{lemma} \label{etaoptimalcond}
The policy $\pi$ is an optimal solution for \eqref{etapi} if and only if 
\begin{equation}\label{Astaropt}
    \mathbb{A}^*_{\pi} \overset{\triangle}{=}\max_{\pi'}\ \mathbb{A}_{\pi}(\pi')=\sum\limits_{s}\rho_{\pi}(s)\max\limits_{\pi'(\cdot|s)}\sum\limits_{a}\pi'(a|s)A_{\pi}(s,a)=0,
\end{equation}
 i.e., $\pi\in\mathrm{argmax}_{\pi'}\ \mathbb{A}_{\pi}(\pi')$.
\end{lemma}

\begin{proof}
Since $\sum\limits_{a}\pi(a|s)A_{\pi}(s,a)=0$ for any $s$ and \eqref{rhocond} holds for $\pi$, it is obviously that the condition \eqref{Astaropt} is equivalent to,   for any policy $\pi'$,
\begin{equation}\label{eq:subopt}
\sum_{a}\pi'(a|s)A_{\pi}(s,a)\leq 0,\ \forall s.
\end{equation}

We first prove the sufficiency part. Combining the conditions \eqref{rhocond} and \eqref{eq:subopt} for any policy $ \pi'$, we have that
\begin{equation*}
\eta(\pi')=\eta(\pi)+\sum_{s}\rho_{\pi'}(s)\sum_{a}\pi'(a|s)A_{\pi}(s,a) \leq \eta(\pi).
\end{equation*}
Hence, $\pi$ is an optimal solution of \eqref{etapi}. 
 
We then prove the necessary part by contradiction. 
Suppose that \eqref{eq:subopt} is not satisfied. Then there exists a state $s’$ and policy $\pi'$ such that 
\[
\sum_{a}\pi'(a|s')A_{\pi}(s’,a) >  0.
\]
Define a new policy $\tilde\pi$ as
\[
\tilde\pi(a |s  ) = 
\begin{cases}
\pi'(a |s  ),&\mbox{}\   s = s', \\
\pi(a |s ), &\mbox{}\   s \neq s'.
\end{cases}
\]
Since \eqref{rhocond} holds for $\tilde{\pi}$,  we obtain   
\begin{eqnarray*}
\eta(\tilde\pi)&=&\eta(\pi)+\sum_{s}\rho_{\tilde \pi}(s)\sum_{a}\tilde \pi(a|s)A_{\pi}(s,a) \\
&=& \eta(\pi) + \rho_{\tilde \pi}( s')\sum_{a} \pi'(a|s')A_{\pi}(s',a) > \eta(\pi),
\end{eqnarray*}
i.e., $\pi$ is not an optimal solution of \eqref{etapi}, which completes the proof.
\end{proof}

We next show a lower bound of improvement for the function  $L$ in each step.
  \begin{lemma}\label{Limprove}
  Suppose $\{\pi_k\}$ is the sequence generated by the trust region method, then we have
 \begin{equation*}
 L_{\pi_k}(\pi_{k+1})-L_{\pi_k}(\pi_k)\geq\min(1,(1-\gamma)\delta_k)\mathbb{A}^*_{\pi_k}.
 \end{equation*}
\end{lemma}
\begin{proof}	
If the optimal solution of subproblem \eqref{trposubanaly}  lies in the trust region, i.e., 
\begin{equation*}\mathbb{E}_{s \sim \rho_{\pi_{k}}} [D_{TV}(\pi_{k}(\cdot |s)||\pi_{k+1}(\cdot |s))]< \delta_k, 
\end{equation*}
we obtain 
\begin{equation*}
\pi_{k+1} = \mathrm{argmax}_{\pi}\ L_{\pi_k}(\pi)= \mathrm{argmax}_{\pi}\ \mathbb{A}_{\pi_{k}}(\pi).
\end{equation*} 
In other words, $L_{\pi_k}(\pi_{k+1})-L_{\pi_k}(\pi_k) =\mathbb{A}^*_{\pi_k}$.

If the optimal solution of subproblem \eqref{trposubanaly} reaches the boundary, it means that $\pi^*_{k+1}=\mathrm{argmax}_{\pi}\ \mathbb{A}_{\pi_{k}}(\pi)$ is outside the trust region. Take $\hat{\pi}_{k+1}$ as a feasible convex combination of $\pi_k$ and $\pi^*_{k+1}$, that is, set $\beta = (1-\gamma)\delta_k$ and 
\begin{equation*}
\hat{\pi}_{k+1} = (1-\beta)\pi_k+\beta\pi^*_{k+1}.
\end{equation*}
Consequently, we have 
\begin{equation*}
\begin{split}
\mathbb{E}_{s \sim \rho_{\pi_{k}}} [D_{TV}(\pi_{k}(\cdot |s)||\hat{\pi}_{k+1}(\cdot |s))] =& \sum_{s}\rho_{\pi_k}(s)D_{TV}(\pi_{k}(\cdot |s)||\hat{\pi}_{k+1}(\cdot |s))\\
=&\beta\sum_{s}\rho_{\pi_k}(s)D_{TV}(\pi_{k}(\cdot |s)||\pi^*_{k+1}(\cdot |s))\\
\leq& \frac{\beta}{1-\gamma}=\delta_k,
\end{split}\end{equation*}
where the inequality is based on the facts that the total variation distance is uniformly bounded by one and $\sum\limits_{s}\rho_{\pi_k}(s)=\frac{1}{1-\gamma}$. Then, it follows that
 \begin{equation*}
 \begin{split}
  L_{\pi_k}({\pi}_{k+1})-L_{\pi_k}(\pi_k)&\geq L_{\pi_k}(\hat{\pi}_{k+1})-L_{\pi_k}(\pi_k)\\
  &=\sum_{s}\rho_{\pi_k}(s)\sum_{a}\left(\hat{\pi}_{k+1}(a|s)-\pi_{k}(a|s)\right)A_{\pi_k}(s,a)\\
  &=\beta\sum_{s}\rho_{\pi_k}(s)\sum_{a}\pi^*_{k+1}(a|s)A_{\pi_k}(s,a)\\
  &=(1-\gamma)\delta_k\mathbb{A}^*_{\pi_k},
\end{split}\end{equation*}
which completes the proof.
\end{proof}	

The next lemma shows a lower bound of the trust region ratio $r_k$. If the trust region radius goes to zero, we have the ratio goes to one.
 \begin{lemma}
The ratio $r_k$ defined in \eqref{ratio} satisfies that
\begin{equation*}
r_k\geq\min\left(1-\frac{4\bar{A}_k\gamma \delta_k^2}{p_0^2(1-\gamma)^2\mathbb{A}^*_{\pi_k}},1-\frac{4\bar{A}_k\gamma \delta_k}{p_0^2(1-\gamma)^3\mathbb{A}^*_{\pi_k}}\right),
\end{equation*}
where $p_0=\min\limits_{s}\rho_0(s)$ and $\bar{A}_k = \max\limits_{s,a}|A_{\pi_k}(s,a)|$.
 \end{lemma}
 \begin{proof}
 It follows from \cite[Theorem 1]{schulman2015trust} that
 \begin{equation}\label{eq:leta}
 \eta({\pi}_{k+1})\geq L_{\pi_k}({\pi}_{k+1})-\frac{4\bar{A}_k\gamma\alpha^2}{(1-\gamma)^2},
 \end{equation}
 where $\alpha = \max\limits_{s}D_{TV}(\pi_{k}(\cdot |s)||{\pi}_{k+1}(\cdot |s))$. Thus,
 \begin{equation*}\begin{split}
 r_{k}&=\frac{\eta({\pi}_{k+1})-\eta(\pi_{k})}{ L_{\pi_k}({\pi}_{k+1})-L_{\pi_k}(\pi_k)}\\
 &\geq 1-\frac{4\bar{A}_k\gamma\alpha^2}{(1-\gamma)^2(L_{\pi_k}({\pi}_{k+1})-L_{\pi_k}(\pi_k))}\\
 &\geq 1-\frac{4\bar{A}_k\gamma\alpha^2}{(1-\gamma)^2\min(1,(1-\gamma)\delta_k)\mathbb{A}^*_{\pi_k} }.
\end{split} \end{equation*}
 Since the relationship
 \begin{equation*}\begin{split}
 \delta_k&\geq\mathbb{E}_{s \sim \rho_{\pi_{k}}} [D_{TV}(\pi_{k}(\cdot |s)||{\pi}_{k+1}(\cdot |s))]\geq p_0\max_{s}[D_{TV}(\pi_{k}(\cdot |s)||{\pi}_{k+1}(\cdot |s))] 
\end{split} \end{equation*}
 holds, i.e., $\alpha\leq\frac{\delta_k}{p_0}$, then we have that
  \begin{equation*}
 r_{k}\geq 1-\frac{4\bar{A}_k\gamma\delta_k^2}{p_0^2(1-\gamma)^2\min(1,(1-\gamma)\delta_k)\mathbb{A}^*_{\pi_k} }.
 \end{equation*}
\end{proof}

In particular, the inequality \eqref{eq:leta} only provides some descent properties of the objective function $\eta$, but it  cannot guarantee the convergence directly.
Finally, we show our main theorem.
\begin{theorem}[Convergence]\label{thm:conv-dis}
    Suppose that $\{\pi_k\}$ is a sequence generated by our trust region method, then we have the following conclusions.
\begin{enumerate}
\item $\underset{k \rightarrow \infty}{\lim\inf}\ \mathbb{A}^*_{\pi_k}=0$.
\item $\lim\limits_{k\to \infty}\eta(\pi_k)=\eta(\pi^*)$, where $\pi^*$ is an optimal solution of \eqref{etapi}.
\end{enumerate}
\end{theorem}
\begin{proof}
We prove the first statement by contradiction. Suppose that $\exists\ \epsilon>0$ and $K\in\mathbb{N}$ such that 
\begin{equation*}
\mathbb{A}_{\pi_k}^*\geq\epsilon,\ \forall k>K.
\end{equation*}
Without loss of generality, we can assume $(1-\gamma)\delta_k<1$ holds for any $k>K$. Then we have 
\begin{equation*}
 L_{\pi_k}({\pi}_{k+1})-L_{\pi_k}(\pi_k)\geq (1-\gamma)\delta_k\epsilon,\\
 r_k\geq1-\frac{4\bar{A}_k\gamma \delta_k}{p_0^2(1-\gamma)^3\epsilon}.
\end{equation*}
Take $\bar{\delta} = \frac{p_0^2(1-\gamma)^3\epsilon(1-\beta_1)}{4\bar{A}\gamma}$, where $\bar{A}\geq\max\limits_{k\geq K}\bar{A}_k$. Once $\delta_k\leq\bar{\delta}$, then
\begin{equation*}
r_{k}\geq\beta_1\mathrm{\ and \ }\delta_{k+1}\geq\delta_k.
\end{equation*}
Hence, we obtain
\begin{equation}\label{deltabound}
\delta_{k}\geq\min(\delta_K,\gamma_2 \bar{\delta}),\ \forall k>K.
\end{equation}
We next claim that $r_k\geq\beta_1$ occurs infinite many times. If not, then there exists $K_1>K$ such that $r_k<\beta_1$ and $\delta_{k+1}<\delta_k$, $\forall k>K_1$, which conflicts with \eqref{deltabound}. Therefore, we have
\begin{equation*}
\eta({\pi}_{k+1})-\eta(\pi_k)=r_k \left(L(\pi_{k+1})-L(\pi_k)\right)\geq(1-\gamma)\delta_k\epsilon\beta_0,\ \forall k> K.
\end{equation*}
Since $\eta(\pi_k)$ is monotone and bounded, we have $\delta_k\rightarrow0$ which is a contradiction.

Denote a subsequence $\{k_{n}\}$ such that $\lim\limits_{n\to \infty}\mathbb{A}^*_{\pi_{k_n}}=0$. Then the continuity of $\mathbb{A}^*_{\pi}$ with respect to $\pi$ indicates that $\lim\limits_{n\to \infty}\pi_{k_n}=\pi^*$. Consequently, the subsequence $\{\eta(\pi_{k_n})\}$ converges since $\eta$ is continuous with respect to $\pi$. Therefore, the full sequence $\{\eta(\pi_{k})\}$ converges since $\{\eta(\pi_k)\}$ is monotone and bounded.
\end{proof}

\subsubsection{Parameterized Policy}
We consider the parameterization of the policy in this part. We mainly focus on the continuous action space and restrict the 
policy into the  Gaussian distribution.
\begin{assumption}\label{assum:gaussian}
    The policy is assumed to be a Gaussian distribution with state-dependent mean vector $\mu(s) \in \mathbb{R}^n$ and a state-independent covariance matrix $\Sigma = \mathrm{diag}(\sigma^2)$, where $\sigma\in\mathbb{R}^n$ is the standard deviation vector and is assumed to be bounded below, i.e. $\sigma^{(i)} \geq \underline\sigma, i = 1,2, \cdots,n.$
\end{assumption}

The parameterization is $\theta=\{\mu(s),\sigma\}_{s \in \mathcal{S}}$ and the policy
$\pi_{\theta_k}(\cdot|s)$ obeys the Gaussian distribution $\mathcal{N}(\mu_k(s),\sigma_k^2)$.
For simplicity, in the following discussion, we often use $\mu$ to represent the notation $\mu(s)$ by ignoring $s$.
The discussion is based on the model \eqref{trposub} with the KL divergence:
\begin{equation}\label{ep-theory}
    \max_{\mu,\sigma}\  L_{\mu_k,\sigma_k}(\mu,\sigma),\ \mathrm{s.t.\ } \mathrm{E}_{s\sim \rho_{\mu_k,\sigma_k}}D_{KL}(\pi_{\mu_k,\sigma_k}(\cdot|s)||\pi_{\mu,\sigma}(\cdot|s))\leq\delta_k.
\end{equation}
Since an exact maximizer of \eqref{ep-theory} may be hard to compute, we use an alternative strategy to get an approximate solution. Firstly, by fixing $\sigma_k$, we update $\mu_k$ by solving the subproblem 
\begin{equation}\label{ep-sigma-theory}
    \max_{\mu}\  L_{\mu_k,\sigma_k}(\mu,\sigma_k),\ \mathrm{s.t.\ } \mathrm{E}_{s\sim \rho_{\mu_k,\sigma_k}}D_{KL}(\pi_{\mu_k,\sigma_k}(\cdot|s)||\pi_{\mu,\sigma_k}(\cdot|s))\leq\delta_k.
\end{equation}
The KL divergence between two Gaussian policies with the same standard deviation can be written as 
\begin{equation*}
    D_{KL}(\pi_{\mu_k,\sigma_k}(\cdot|s)||\pi_{\mu,\sigma_k}(\cdot|s)) = \frac{1}{2}(\mu-\mu_k)^T\Sigma_k^{-1}(\mu-\mu_k).
\end{equation*} 
Then, we update $\sigma_k$ from
\begin{equation}\label{ep-sigma}
    \max_{\sigma}\  L_{\mu_k,\sigma_k}(\mu_{k+1},\sigma),\ \mathrm{s.t.\ } \mathrm{E}_{s\sim \rho_{\mu_k,\sigma_k}}D_{KL}(\pi_{\mu_k,\sigma_k}(\cdot|s)||\pi_{\mu_{k+1},\sigma}(\cdot|s))\leq\delta_k.
\end{equation}
We also make some assumptions on the smoothness of the function $L$.
\begin{assumption}\label{assum:hessen}
    The Hessian matrix $\nabla_\mu^2L_{\mu_k,\sigma_k}(\mu,\sigma_k)$ is uniformly bounded within the trust region, i.e., $h_k = \max\limits_{\mu\in\mathcal{B}_k}\|\nabla_\mu^2L_{\mu,\sigma_k}(\mu,\sigma_k)\|+1 \leq h$ for $\forall k$, where $\mathcal{B}_k=\{\mu| \mathrm{E}_{s\sim \rho_{\mu_k,\sigma_k}}D_{KL}(\pi_{\mu_k,\sigma_k}(\cdot|s)||\pi_{\mu,\sigma_k}(\cdot|s))\leq\delta_k\}$.
\end{assumption}
 
 The next lemma shows that the solution of the subproblem \eqref{ep-sigma-theory} provides a lower bounded model improvement.
\begin{lemma}\label{delta_Lest}
    Suppose that Assumptions \ref{assum:gaussian}-\ref{assum:hessen} hold. The maximizer of \eqref{ep-sigma-theory} satisfies that
\begin{equation}\label{delta_Lestimation}
L_{\mu_k,\sigma_k}(\mu_{k+1},\sigma_k)-L_{\mu_k,\sigma_k}(\mu_{k},\sigma_k)\geq \|g_k\|^2\min\left(\frac{1}{2h_k},\frac{\underline\sigma\sqrt{\delta_k}}{2\|g_k\|}\right),
\end{equation}
where $g_k = \nabla_\mu L_{\mu_k,\sigma_k}(\mu_{k},\sigma_k)$.
\end{lemma}
\begin{proof}
    Take $\mu_{k+1}^c:=\mu_k+\tau g_k$, where $\tau \in(0,\frac{\underline\sigma\sqrt{\delta_k}}{\|g_k\|}]$. Then we have $\mu_{k+1}^c\in \mathcal{B}_k$ and 
\begin{equation}\label{lower1}
\begin{split}
 L_{\mu_k,\sigma_k}(\mu^c_{k+1},\sigma_k) = &L_{\mu_k,\sigma_k}(\mu_{k},\sigma_k)+\tau \|g_k\|^2 +\frac{1}{2}\tau^2\iprod{g_k}{\nabla^2L_{\mu_k,\sigma_k}(\hat{\mu}_{k+1})g_k}\\
\geq& L_{\mu_k,\sigma_k}(\mu_{k},\sigma_k)+(\tau-\frac{1}{2}\tau^2 h_k) \|g_k\|^2,
\end{split}
\end{equation}
where $\hat{\mu}_{k+1}\in \mathcal{B}_k$ lies between $\mu_k$ and $\mu^c_{k+1}$. If $\frac{1}{h_k} \leq \frac{\underline\sigma\sqrt{\delta_k}}{\|g_k\|} $, by taking $\tau = \frac{1}{h_k}$, we can obtain  
\begin{equation}\label{j=0}
  L_{\mu_k,\sigma_k}(\mu^c_{k+1},\sigma_k)-L_{\mu_k,\sigma_k}(\mu_{k},\sigma_k)\geq\frac{\|g_k\|^2}{2h_k}.
\end{equation}
Otherwise $\frac{1}{h_k} > \frac{\underline\sigma\sqrt{\delta_k}}{\|g_k\|} $, we set $\tau = \frac{\underline\sigma\sqrt{\delta_k}}{\|g_k\|}$ and get  
\begin{equation}\label{j>0}
    L_{\mu_k,\sigma_k}(\mu^c_{k+1},\sigma_k)-L_{\mu_k,\sigma_k}(\mu_{k},\sigma_k)\geq\underline\sigma\sqrt{\delta_k}\|g_k\| - \frac{1}{2} \underline\sigma^2\delta_k h_k \geq \frac{1}{2}\underline\sigma\sqrt{\delta_k}\|g_k\| .
\end{equation}
It follows from \eqref{j=0} and \eqref{j>0} that at the point $\mu_{k+1}^c$,
\begin{equation}
L_{\mu_k,\sigma_k}(\mu^c_{k+1},\sigma_k)-L_{\mu_k,\sigma_k}(\mu_{k},\sigma_k)\geq \|g_k\|^2\min\left(\frac{1}{2h_k},\frac{\underline\sigma\sqrt{\delta_k}}{2\|g_k\|}\right),
\end{equation}
which completes the proof.
\end{proof}

Then we show a lower bound of the trust region ratio. 
\begin{lemma}\label{rklowerbound}
    Suppose that Assumptions \ref{assum:gaussian}-\ref{assum:hessen} hold. 
    The ratio defined in \eqref{ratio} satisfies that
\begin{equation}
r_k\geq1-\frac{8\bar{A}_k\gamma\delta_k^2\max(\frac{\|g_k\|}{\underline\sigma_k\sqrt{\delta_k}},h_k)}{p_0^2(1-\gamma)^2\|g_k\|^2},
\end{equation}
where $p_0=\min\limits_{s}\rho_0(s)$ and $\bar{A}_k = \max\limits_{s,a}|A_{\mu_k,\sigma_k}(s,a)|$.
\end{lemma}
\begin{proof}
The explicit formulation of the ratio yields 
\begin{equation}
\begin{split}
r_k=&\frac{\eta(\mu_{k+1},\sigma_{k+1})-\eta(\mu_{k},\sigma_{k})}{L_{\mu_k,\sigma_k}(\mu_{k+1},\sigma_{k+1})-L_{\mu_k,\sigma_k}(\mu_{k},\sigma_{k})}\\
=&\frac{\eta(\mu_{k+1},\sigma_{k+1})-L_{\mu_k,\sigma_k}(\mu_{k+1},\sigma_{k+1})+L_{\mu_k,\sigma_k}(\mu_{k+1},\sigma_{k+1})-\eta(\mu_{k},\sigma_{k})}{L_{\mu_k,\sigma_k}(\mu_{k+1},\sigma_{k+1})-L_{\mu_k,\sigma_k}(\mu_{k},\sigma_{k})}\\
\geq&\frac{-\frac{4\bar{A}_k\gamma\delta_k^2}{p_0^2(1-\gamma)^2}+L_{\mu_k,\sigma_k}(\mu_{k+1},\sigma_{k+1})-L_{\mu_k,\sigma_k}(\mu_{k},\sigma_{k})}{L_{\mu_k,\sigma_k}(\mu_{k+1},\sigma_{k+1})-L_{\mu_k,\sigma_k}(\mu_{k},\sigma_{k})}\\
\geq&1-\frac{4\bar{A}_k\gamma\delta_k^2}{p_0^2(1-\gamma)^2 \|g_k\|^2\min\left(\frac{1}{2h_k},\frac{\underline\sigma\sqrt{\delta_k}}{2\|g_k\|}\right)}\\
\geq&1-\frac{8\bar{A}_k\gamma\delta_k^2\max(\frac{\|g_k\|}{\underline\sigma\sqrt{\delta_k}},h_k)}{p_0^2(1-\gamma)^2\|g_k\|^2},
\end{split}
\end{equation}
where the first inequality follows from \cite[Theorem 1]{schulman2015trust} and $\eta(\mu_{k},\sigma_{k})=L_{\mu_k,\sigma_k}(\mu_{k},\sigma_{k})$, and the second inequality is derived from Lemma \ref{delta_Lest} and Assumption \ref{assum:gaussian}.
\end{proof}

Finally, we establish our main convergence result. As for the Gaussian policies, the optimal policy is supposed to be deterministic, i.e., the mean vector is the optimal action at each state and the standard deviation is zero. It implies that the best action is usually grasped by the mean parameter and the standard deviation plays a role in exploration. Hence, in pratice, we only care the learn of the mean vector and make the final standard deviation small enough. Therefore, we concentrate on the behavior of $\|\nabla_\mu\eta(\mu,\sigma)\|$.
\begin{theorem}
Suppose that Assumptions \ref{assum:gaussian}-\ref{assum:hessen} hold and $\{(\mu_k, \sigma_k)\}$ is the sequence by the trust region method. Then we have the following conclusions.
\begin{enumerate}
    \item The limit inferior of the norm of gradient with respect to $\mu$ goes to zero, i.e., 
\begin{equation*}
\underset{k \rightarrow \infty}{\lim\inf}\ \|\nabla_\mu\eta(\mu_k,\sigma_k)\|=0.
\end{equation*}
\item Suppose that the gradient $\nabla_\mu\eta(\mu,\sigma)$ is Lipschitz continuous with respect to $\mu$ and $\sigma$, respectively. If the standard deviation vector $\sigma_k$ converges to $\sigma^*$, i.e. $\lim\limits_{k \rightarrow \infty}\sigma_k = \sigma^*$, then we have 
\begin{equation*}
\lim\limits_{k \rightarrow \infty}\|\nabla_\mu\eta(\mu_k,\sigma^*)\|=0.
\end{equation*}
\end{enumerate}
\end{theorem} 
\begin{proof}
We prove the first statement by contradiction. Suppose that we can find $\epsilon>0$ and $K\in\mathbb{N}$ such that 
\begin{equation}\label{eq:geps}
\|\nabla_\mu\eta(\mu_k,\sigma_k)\|=\|g_k\|\geq\epsilon,\forall k> K.
\end{equation}
Combining \eqref{eq:geps} with the result in Lemma \ref{rklowerbound}, we have 
\begin{equation*}
    \begin{split}
r_k 
&\geq1-\frac{8\bar{A}_k\gamma\delta_k^2\max(\frac{\|g_k\|}{\underline\sigma\sqrt{\delta_k}},h_k)}{p_0^2(1-\gamma)^2\|g_k\|^2} \\
&\geq \min\left(1-\frac{8\bar{A}_k\gamma\delta_k^{\frac{3}{2}}}{p_0^2(1-\gamma)^2\|g_k\|\underline\sigma}, 
1-\frac{8\bar{A}_k\gamma\delta_k^2h}{p_0^2(1-\gamma)^2\|g_k\|^2}\right) \\
&\geq \min\left(1-\frac{8\bar{A}\gamma\delta_k^{\frac{3}{2}}}{p_0^2(1-\gamma)^2\epsilon\underline\sigma}, 
1-\frac{8\bar{A}\gamma\delta_k^2h}{p_0^2(1-\gamma)^2\epsilon^2}\right)
\end{split}
\end{equation*}
where $\bar{A}\geq\max\limits_{k\geq K}\bar{A}_k$. Take 
\begin{equation*}
\bar{\delta} = \min\left(\left(\frac{p_0^2(1-\beta_1)(1-\gamma)^2\bar{\sigma}\epsilon}{8\bar{A}\gamma}\right)^{\frac{2}{3}},
\left(\frac{p_0^2(1-\beta_1)(1-\gamma)^2\epsilon^2}{8\bar{A}\gamma h}\right)^{\frac{1}{2}}\right).
\end{equation*}
It follows that for each $\delta_k\leq\bar{\delta}$, the ratio $r_k\geq\beta_1$.
According to the update rule \eqref{updateC}, we have
\begin{equation}\label{delta_klowerbound}
\delta_k\geq\min(\delta_K,\gamma_2\bar{\delta}), \forall k>K.
\end{equation}
Now, we claim that $r_k\geq\beta_1$ occurs infinite many time. Otherwise, we can find $K_1>K$ such that for any $k>K_1$, $r_k<\beta_1$ and $\delta_k>\delta_{k+1}$, which conflicts with \eqref{delta_klowerbound}. Therefore, we can obtain for $\forall k>K$ and $r_k \geq \beta_0$,
\begin{equation*}
\begin{split}
\eta(\mu_{k+1},\sigma_{k+1})-\eta(\mu_{k},\sigma_{k})
&= r_{k}\left( L_{\mu_k,\sigma_k}(\mu_{k+1},\sigma_{k+1})-L_{\mu_k,\sigma_k}(\mu_{k},\sigma_{k})\right)\\
&\geq r_{k}\left( L_{\mu_k,\sigma_k}(\mu_{k+1},\sigma_{k})-L_{\mu_k,\sigma_k}(\mu_{k},\sigma_{k})\right)\\
&\geq\beta_0\|g_k\|^2\min\left(\frac{1}{2h},\frac{\underline\sigma\sqrt{\delta_k}}{2\|g_k\|}\right)\\
&\geq\beta_0\min\left(\frac{\epsilon^2}{2h},\frac{\underline\sigma\epsilon\sqrt{\delta_k}}{2}\right), \\
\end{split}
\end{equation*}
where the first inequality is due to the definition of the trust region ratio, the second inequality comes from the update rule of $\sigma$ \eqref{ep-sigma} and the third inequality is from Lemma \eqref{delta_Lest}. Since $\eta$ is continuous and upper bounded, we have $\delta_{k}\rightarrow 0$ as $k\rightarrow\infty$ which contradicts to \eqref{delta_klowerbound}.

In order to construct a contradiction for the second statement, we assume that there exists some $\epsilon>0$ such that $\|g_{k_i}\|\geq2\epsilon$, where $\{k_i\}$ is a subsequence of successful iterates. Denote the indices of all successful iterates by $\mathcal{M}$. From the first statement and its proof, we can obtain another subsequence of $\mathcal{M}$, denoted by $\{t_i\}$, where $t_i$ is the first successful iterate such that $t_i>k_i$ and $\|g_{t_i}\|\leq\epsilon$. Let $\mathcal{Q}=\{q\in\mathcal{M}|k_i\leq q<t_i, \mathrm{for\ some\ }i =1,2,...\}$. Then for any $q\in\mathcal{Q}$ we have
\begin{equation*}
\eta(\mu_{q+1},\sigma_{q+1})-\eta(\mu_{q},\sigma_{q})\geq\beta_0\min\left(\frac{\epsilon^2}{2h},\frac{\underline\sigma\epsilon\sqrt{\delta_q}}{2}\right).
\end{equation*}
Obviously, the sequence $\{\eta(\mu_{k},\sigma_{k})\}_{k \in \mathcal{M}}$ is monotonically increasing and upper bounded. Hence, we obtain 
\begin{equation*}
\lim_{\begin{subarray}{c} q\in\mathcal{Q}\\ q\rightarrow\infty\end{subarray}}\delta_q=0, \lim_{\begin{subarray}{c} q\in\mathcal{Q}\\ q\rightarrow\infty\end{subarray}}\eta(\mu_q,\sigma_q)=\eta^*.
\end{equation*}
Without loss of generality, we can assume $\frac{\underline\sigma\sqrt{\delta_q}}{2}\leq\frac{\epsilon}{2h}$ for all $q\in\mathcal{Q}$. Then for each $i$, we have

\begin{eqnarray*}
\eta^*-\eta(\mu_{k_i},\sigma_{k_i}) &\geq&
\eta(\mu_{t_i},\sigma_{t_i})-\eta(\mu_{k_i},\sigma_{k_i})\geq\sum\limits_{\begin{subarray}{c}q\in\mathcal{Q}, \\
k_i\leq q<t_i\end{subarray}}\beta_0\frac{\underline\sigma\epsilon\sqrt{\delta_q}}{2}  \\
&\geq& \frac{\beta_0\underline\sigma\epsilon }{2\bar\sigma} \sum\limits_{\begin{subarray}{c}q\in\mathcal{Q}, \\
k_i\leq q<t_i\end{subarray}} \|\mu_{q+1}-\mu_{q}\|  \geq \frac{\beta_0\underline\sigma\epsilon }{2\bar\sigma} \|\mu_{t_i}-\mu_{k_i}\|,
\end{eqnarray*}
where $\bar\sigma$ is a upper bound of each element in $\sigma_k$ which exists due to the convergence of $\sigma_k$.
Thus, we derive $\|\mu_{t_i}-\mu_{k_i}\|\rightarrow0$, and consequently $\|g_{t_i}-g_{k_i}\|\rightarrow0$, which yields a contradiction to $\|g_{t_i} - g_{k_i}\| \geq \epsilon$ because of  the definitions of $\{k_i\}$ and $\{t_i\}$.
It concludes that 
\[
\lim\limits_{k \rightarrow \infty}\|\nabla_\mu\eta(\mu_k,\sigma_k)\|=0.
\]
Then we have
\[
\|\nabla_\mu\eta(\mu_k,\sigma^*)\| \leq \|\nabla_\mu\eta(\mu_k,\sigma^*) - \nabla_\mu\eta(\mu_k,\sigma_k)\| + \|\nabla_\mu\eta(\mu_k,\sigma_k)\| \rightarrow 0,
 \]
 where the first term goes to zero due to the convergence of $\sigma_k$ and the Lipschitz continuity of $\nabla \eta$.
\end{proof}

\section{A Stochastic Trust Region Algorithm}\label{Empirical Algorithm}
The functions discussed before are constructed with expectations and they have to be estimated using sample averages in practice. To be consistent with the notations above, we add a hat to represent the corresponding estimated value, such as $\hat{\eta}$, $\hat{L}_{\theta_k}$ and so on. Essentially, the optimization in Algorithm \ref{algo-expectation-form} is constructed with these estimated functions. In consideration of the uncertainty and fluctuations caused by the samples, we develop a stochastic version of Algorithm \ref{algo-expectation-form} in the following discussion.
\subsection{Sample Collection and Advantage Estimation}
The original objective function $\eta(\theta)$ denotes the expectation of the cumulative discounted rewards over a trajectory generated by the policy $\pi_{\theta}$. Practically, we simulate $N$ interactions with the environment from policy $\pi_{\theta_{k}}$ to generate $\tau$ complete trajectories whose total rewards are denoted as $\{R_{i}\}_{i=1}^{\tau}$, using Algorithm \ref{sample}. All visited states and actions during the interactions are gathered into a set of sample pairs $B_{\theta_k}=\{(s_i,a_i)\}_{i=1}^N$ for estimation. The averaged total reward of these trajectories: 
\begin{equation}
\hat{\eta}(\theta_{k}) = \frac{1}{\tau}\sum_{i=1}^{\tau}R_{i}
\end{equation}
is called as an empirical estimation of $\eta(\theta_{k})$. The sample standard deviation
\begin{equation}
\hat{\sigma}_{\eta}(\theta_{k})=\sqrt{\frac{1}{\tau-1}\sum_{i=1}^{\tau}(R_{i}-\hat{\eta}(\theta_{k}))^{2}}
\end{equation}
characterizes the sample fluctuation near the mean value.  The expectations can be approximated with respect to the samples as follows:
\begin{equation}\label{empirform}
\begin{split}
\hat{L}_{\theta_k}(\theta,B_{\theta_k})=&\hat{\eta}(\theta_{k}) +\frac{1}{|B_{\theta_k}|}\sum\limits_{(s,a)\in B_{\theta_k}}\frac{\pi_{\theta}(a|s)}{\pi_{\theta_{k}}(a|s)}\hat{A}_{\theta_{k}}(s,a),\\
 \hat{g}_k(\theta,B_{\theta_k})=&\frac{1}{|B_{\theta_k}|}\sum\limits_{(s,a)\in B_{\theta_k}}\frac{\nabla\pi_{\theta}(a|s)}{\pi_{\theta_{k}}(a|s)}\hat{A}_{\theta_{k}}(s,a),\\
\hat{D}_{k}(\theta,B_{\theta_k})=& \frac{1}{|B_{\theta_k}|}\sum\limits_{(s,a)\in B_{\theta_k}}D(\pi_{\theta_{k}}(\cdot|s),\pi_{\theta}(\cdot|s)),
\end{split}
\end{equation}
where the advantage estimator $\hat{A}_{\theta_{k}}$ is constructed using the empirical reward and the value network. We take a version of GAE as: 
\begin{equation}\label{advancomput}
\hat{A}_{\theta_k}(s, a)=r(s,a)+ \gamma V_{\phi_{k}}(s')-V_{\phi_{k}}(s) +\gamma\lambda\hat{A}_{\theta_k}(s', a'),
\end{equation}
where $V_\phi$ is the state value function parameterized by $\phi$ and $s'$ is the state after the transfer from $s$ and $a$. The hyper-parameter $\lambda$ controls the trade-off between bias and variance. The value network is trained by minimizing the error between the target value and the predicted value:
 \begin{equation}\label{critic-prob}
\phi_{k+1} = \mathop{\arg\min}\limits_{\phi}\frac{1}{| B_{\theta_k}|}\sum\limits_{(s,a)\in B_{\theta_k}}\rVert V_{\phi_{k}}(s)+\hat{A}_{\theta_k}(s,a)- V_{\phi}(s) \rVert^2.
 \end{equation} 
 Fitting such a value network as the baseline function is the same as the other policy-based methods in deep reinforcement learning \cite{ilyas2018deep}.

\begin{algorithm}[tbp]
		\caption{SAMPLE($\theta$,$N$)}
		\label{sample}
		\begin{algorithmic}[1]
\STATE $B_{\theta}=\emptyset$; $\tau = 0$, $R_{\tau} = 0,t=0$; randomly initiate the state $s_0$;
			\FOR{$i=1,\cdots,N$}
			\IF {$s_t$ is a terminal state}
			\STATE $\tau = \tau+1$; $R_{\tau} = 0,t=0$; randomly initiate the state $s_0$;
			\ELSE 
			\STATE perform one step at $s_t$ using $\pi_{\theta}(\cdot|s_t)$ and obtain reward $r_{t}:=r(s_{t},a_{t},s_{t+1})$;
			\STATE $R_{\tau} = R_{\tau}+\gamma^{t}r_{t}$;
			\STATE $t=t+1$;
			\STATE $B_{\theta}\leftarrow B_{\theta}\cup (s_{t},a_{t})$;
			\ENDIF
		\ENDFOR	
		\STATE $\hat{\eta}(\theta)= \frac{1}{\tau}\sum_{j=1}^{\tau}R_{j}$;
	\RETURN $B_{\theta}$		
		\end{algorithmic}
	\end{algorithm}

\subsection{Solving the Trust Region Subproblem}\label{subsect}
Empirically, the trust region model for policy optimization  at the $k$-th iteration is 
\begin{equation}\label{empiricalobj}
\max_{\theta}\quad\hat{L}_{\theta_k}(\theta,B_{\theta_k}),\ \mathrm{s.t.\ } \hat{D}_{k}(\theta,B_{\theta_k})\leq\delta_k.
\end{equation}
It is worth noting that the policy distance function in \eqref{empiricalobj} can be approximately treated as a re-weighted matrix norm on parameters. In other words, it holds for any metric function $D$ and sample set $B$:
\begin{equation*}
\hat{D}_{k}(\theta,B)\approx \frac{1}{2}(\theta-\theta_k)^T\nabla^2\hat{D}_{k}(\theta_k,B)(\theta-\theta_k),
\end{equation*}
since  $\hat{D}_{k}(\theta,B)\geq0$, $\nabla^2\hat{D}_{k}(\theta_k,B)$ is semi-positive definite, and the first-order Taylor expansion at $\theta=\theta_k$ is equal to zero. 

We now present a  feasible stochastic  method for solving \eqref{empiricalobj} by sequentially generating increasing directions using the CG algorithm. The  process starts from $\theta_{k,1}=\theta_k$. At the $l$-th inner iteration,  we randomly sample a subset $b_l\subseteq B_{\theta_k}$  to obtain $\hat{g}_k(\theta_{k,l},b_l)$ and $\hat{H}_k(b_l)=\nabla^2\hat{D}_{k}(\theta_k,b_l)$, 
and  compute the direction
\[ d_l=\hat{H}_k^{-1}(b_l)\hat{g}_k(\theta_{k,l},b_l).\]
Then  we update the parameter as
\begin{equation} \theta_{k,l+1} =\theta_{k,l}+\alpha d_l
\end{equation}
such that the following two conditions are satisfied:
\begin{equation} \label{lineserach-cond}
\begin{split}
\hat{L}_{\theta_{k}}(\theta_{k,l+1},b_l)\geq \hat{L}_{\theta_{k}}(\theta_{k,l},b_l)+ \tau d_l ^T\hat{g}_{k}(\theta_{k,l},b_l),\mathrm{and}\ \hat{D}_k(\theta_{k,l+1},B_{\theta_k}) \leq \delta_k,
\end{split}
\end{equation}
where $\tau\in(0,1)$. The conditions in \eqref{lineserach-cond} ensure an improvement of $\hat{L}_{\theta_k}$ on $b_l$ and the constraint in \eqref{empiricalobj} holds at each step. 

The theoretical analysis in Lemma \ref{Limprove} suggests that the improvement of $L_{\theta_k}$ is supposed to be larger than $(1-\gamma)\delta_k A_{\theta_k}^*$. However, it is unaccessible in practice. Empirically, we terminate the iteration and denote the trial point $\tilde{\theta}_{k+1}:=\theta_{k,l+1}$ once the following conditions hold:
\begin{equation}\label{emp-innerterminate}
\frac{|\hat{L}_{\theta_{k}}(\theta_{k,l+1},B_{\theta_k})-\hat{L}_{\theta_{k}}(\theta_{k,l},B_{\theta_k})|}{1+\left|\hat{L}_{\theta_{k}}(\theta_{k,l},B_{\theta_k})\right|}\leq\epsilon\mathrm{\ or\ }  \frac{\left|\mathrm{Ent}(\theta_{k,l+1},B_{\theta_k})-\mathrm{Ent}(\theta_{k},B_{\theta_k})\right|}{1+\left|\mathrm{Ent}(\theta_{k},B_{\theta_k})\right|}\geq\epsilon, 
\end{equation}
where $\epsilon>0$ is a small constant and $\mathrm{Ent}(\theta_k, B_{\theta_k})$ is the sample-averaged entropy of the policy $\pi_{\theta_k}$, i.e., 
\begin{equation*}
\mathrm{Ent}(\theta_k, B_{\theta_k})=\frac{1}{|B_{\theta_k}|}\sum\limits_{s\in B_{\theta_k}}\mathcal{H}(\pi_{\theta_k}(\cdot|s))\mathrm{\ and\ } \mathcal{H}(p) = -\int p(x)\log p(x)dx .
\end{equation*}
The entropy characterizes the randomness and the exploration ability of the policy. Under the random circumstances, the policy inadvertently collapses on the best observed action with respect to the estimated advantage value which may be significantly wrong. Although the distance constraint recedes this effect to some extent, the policy is driven to favor the best actions it has visited, albeit slowly \cite{abdolmaleki2018relative}. Based on these concerns, we take the second condition in \eqref{emp-innerterminate} to prevent excessively dependency on the advantage estimators and limited observations. 
The pseudocode for solving \eqref{empiricalobj} is outlined in Algorithm \ref{multi-trpo}.

\begin{algorithm}[tbp]
		\caption{InnerSolu($\theta_k$,$\delta_k$,$B_{\theta_k}$,$\xi$,$\tau$)}
		\label{multi-trpo}
		\begin{algorithmic}[1]
\STATE $\theta_{k,1} = \theta_k$;
		\FOR{$l=1,2,...$}
			\STATE randomly sample data $ b_l\subseteq B_{\theta_k}$;
			\STATE compute $\hat{g}_k(\theta_{k,l},b_l)$ and $\hat{H}_k(b_l)$ by \eqref{empirform};
			\STATE compute $d_l=\hat{H}_k^{-1}(b_l)\hat{g}_k(\theta_{k,l},b_l)$ using the CG algorithm;			
			\STATE set $\theta_{k,l+1} = \theta_{k,l}+\alpha d_l$ such that \eqref{lineserach-cond} holds;
\IF{$l\ \%\ l_0=0$}
\STATE If \eqref{emp-innerterminate} holds, break;
\ENDIF
\ENDFOR
	
\RETURN $\tilde{\theta}_{k+1}=\theta_{k,l+1}$	
		\end{algorithmic}
	\end{algorithm}

\subsection{Alternating Strategy For Gaussian Policy}	\label{Separating}
For continuous action spaces, the multivariate Gaussian distribution is adopted in many cases. The neural network whose weight is denoted as $\theta^{\mu}$ maps from state to the mean of the Gaussian distribution, $\mu(s;\theta^\mu)$. The covariance matrix is assumed to be a diagonal matrix. Namely, we take a state-independent vector $\theta^{\sigma}$ to represent the log-standard deviations instead of the standard deviations, since the log-standard deviations are free to take values from $(-\infty,\infty)$. Therefore, the policy is characterized by the Gaussian distribution 
\begin{equation*}
\pi_{\theta}(\cdot|s)\sim N(\mu(s;\theta^\mu), \mathrm{diag}(\exp(2\theta^{\sigma}))),\forall s\in\mathcal{S},
\end{equation*}
where $\theta = (\theta^\mu,\theta^{\sigma})$. It can be verified that the action can be characterized as: 
\begin{equation}\label{sampleaction}
a = \mu(s;\theta^\mu)+ \exp(\theta^{\sigma})\odot \varepsilon, \varepsilon\sim N(0,I),
\end{equation}
where $\odot$ is the element-wise product of two vectors. Ideally, the optimal Gaussian policy is supposed to be deterministic, that is, the mean is the optimal action and the covariance is zero. 

As shown in \eqref{sampleaction}, the covariance term affects the exploration of the samples as well as the quality of the optimization. An optimistic situation is that the log-standard deviations gradually decreases with the iterative process and is finally controlled by a lower bound, since extremely small covariance is undesirable for numerical stability in practice. In order to reduce the interactions between the mean and log-standard deviations in optimization and regulate the decay of the log-standard deviations more precisely, we can approximately solve the trust region problem \eqref{empiricalobj} by alternating the update with respect to the mean and log-standard deviations independently. For the mean parameter, we formulate the trust region model as follows:
\begin{equation}\label{sep-gaussian}
\max_{\theta^\mu} \ \hat{L}_{\theta_k}\left(\theta = (\theta^\mu,\theta_{k}^\sigma), B_{\theta_k}\right),\ \mathrm{s.t.\ } \hat{D}_k(\theta, B_{\theta_k})\leq \delta_k,
\end{equation}
As we assumed in the theoretical analysis, the update of the log-standard deviations is supposed to monotonically improve $L_{\theta_k}$ and ensure that the new policy is close to the current policy. For simplicity, we can update $\theta^\sigma$ by approximately solving
\begin{equation}\label{sigma-general}
\max_{\theta^\sigma} \ \hat{L}_{\theta_k}\left(\theta = (\tilde{\theta}_{k+1}^\mu,\theta^\sigma), B_{\theta_k}\right),\ \mathrm{s.t.\ } \|\theta^{\sigma}_k-\theta^\sigma\|_{\infty}\leq\bar{A}_k,
\end{equation}
where $\tilde{\theta}_{k+1}^\mu$ is the solution of \eqref{sep-gaussian}. Then we denote the solution of \eqref{sigma-general} as $\tilde{\theta}_{k+1}^\sigma$ and the trial point $\tilde{\theta}_{k+1}=(\tilde{\theta}_{k+1}^\mu,\tilde{\theta}_{k+1}^\sigma)$.

Essentially, the infinite norm constraint controls the distance between two Gaussian policies and plays a similar role in regulating the entropy loss for the Gaussian policy since 
\begin{equation*}
\mathrm{Ent}(\theta_k,B_{\theta_k})=1^T\theta_k^\sigma+\frac{n}{2}(1+\log2\pi),
\end{equation*}
 where $n$ is the dimension of the action space. We can take the bound $\bar{A}_k$ related to the entropy of the current policy, i.e., $\epsilon_{k}\propto |\mathrm{Ent}(\theta_k, B_{\theta_k})|$. The subproblem \eqref{sep-gaussian} is solved by following the process in Algorithm \ref{multi-trpo}, while the solution of the subproblem \eqref{sigma-general} is approximated using the projected gradient method. In our experiments, we find that the alternative update for the Gaussian policy is helpful in releasing the log-standard deviations from the extreme decline caused by the estimations and encouraging the mean parameter to update toward the accurate optimal actions.

\subsection{Algorithmic Development}	
According to the Algorithm \ref{algo-expectation-form}, if the iteration is successful, we exploit the samples for testing $\tilde{\theta}_{k+1}$ in the next iteration without additional simulations. However, once a rejection of $\tilde{\theta}_{k+1}$ arises, the policy remains unchanged, i.e., $\theta_{k+1}=\theta_k$. Apparently, the samples in $B_{\theta_k}$ can be used to construct the trust region model \eqref{empiricalobj} in the next iteration. In such a situation, apart from updating the trust region radius $\delta_{k}$ using \eqref{updateC}, we simulate another $N$ sample pairs from $\pi_{\theta_{k+1}}$ and merge it with $B_{\theta_k}$ to be $B_{\theta_{k+1}}$. Thereby after a rejection,  the sample size is enlarged and the estimations are supposed to be more accurate. Ideally, when the sample size is sufficiently large and the trust region gets small enough, an acceptance is very likely to occur with high probability. In consideration of data storage, once the sample size reaches an upper limit $N_{\max}$ after several consecutive rejections, a mandatory acceptance is enforced such that the best performed policy among the last few rejected iterations is taken as the next iteration:
\begin{equation}\label{forceaccept}
(\theta_{k+1},B_{\theta_{k+1}}) = \mathop{\arg\max}\limits_{(z,T_z)\in H}\ \hat{\eta}(z),
\end{equation}
where $H$ is the set of rejected iterations and $T_z$ is the sample set with respect to the policy $\pi_{z}$.

Undesirably, the landscape of $\eta$ becomes much more noisy in low-sampled regime as the training progresses \cite{ilyas2018deep}. More seriously,  the total expected rewards $\eta(\theta_{k})$ and $\eta(\tilde{\theta}_{k+1})$ are estimated from different sample sets, which is different from that of the general deep learning problems. These challenges make the estimations to be accompanied with large variance. Therefore, we take the sample standard deviation into the ratio \eqref{ratio} for empirical stability:
\begin{equation}\label{ratio-empir}
r_{k} = \frac{\hat{\eta}(\tilde{\theta}_{k+1})-\hat{\eta}(\theta_{k})}{\hat{\sigma}_{\eta}(\theta_{k})+\hat{L}_{\theta_k}(\tilde{\theta}_{k+1},B_k)-\hat{L}_{\theta_k}(\theta_{k},B_k)}.
\end{equation}
As a generalization of the standard ratio in trust region methods, this revision checks the agreement between the objective function and surrogate function in a confidence interval. 

In view of the oscillations under sampling, a small negative ratio is permitted in acceptance. Namely the update criterion for $\theta_{k+1}$ is modified from \eqref{updatetheta} as:
\begin{equation}\label{updatetheta-empir}
\theta_{k+1}=\begin{cases}\tilde{\theta}_{k+1},&\mbox{}\  r_{k}\geq0>\beta_0,\\
\theta_{k},&\mbox{}\ r_{k}\leq\beta_0\ \mathrm{and\ }|B_{\theta_k}|< N_{\max},\\
\hat{\theta}_{k+1},&\mbox{}\ \mathrm{otherwise.}
\end{cases}
\end{equation}
We select the trust region radius depending linearly on the norm of the gradient \cite{wang2019stochastic}, i.e., $\delta_k =\mu_k \|g_{k}(\theta_{k})\|$. The adjustment of the coefficient $\mu_{k}$ is based on the ratio as:
\begin{equation}\label{updateC-empir}
\mu_{k+1}=\begin{cases}\min(\gamma_1\mu_{k},\mu_{\max}),&\mbox{}\  r_{k}\geq\beta_1,\\
\max(\gamma_2\mu_{k},\mu_{\min}),&\mbox{}\  r_{k}\in[\beta_0,\beta_1),\\
\max(\gamma_3\mu_{k},\mu_{\min}),&\mbox{}\ \mathrm{otherwise,}
\end{cases}
\end{equation}
where $\beta_1\geq0>\beta_0$ and $\mu_{\max}>\mu_{\min}>0$. The empirical algorithm is summarized in Algorithm \ref{sprlalgo}.
\begin{algorithm}[tbp]
		\caption{STRO}
		\label{sprlalgo}
		\begin{algorithmic}[1]
			\REQUIRE Set $\mu_{0}$, $\theta_0$, $\phi_0$, $N$, $N_{max}$, $\sigma$, $\tau$.
\STATE $k=0$, $B_{\theta_k}=\mathrm{SAMPLE}(\theta_{k},N)$, $H =\emptyset$; 
			\WHILE{$\mathrm{stopping\ criterion\ not\ met}$}	
			\STATE compute a trial point $\tilde{\theta}_{k+1} = \mathrm{InnerSolu}(\theta_k,\delta_k,B_{\theta_k},\sigma,\tau)$;
			 \STATE $T_{\tilde{\theta}_{k+1}} = \mathrm{SAMPLE}(\tilde{\theta}_{k+1},N)$;
			\STATE compute the ratio $r_{k}$ via \eqref{ratio-empir};
				\STATE  update $\mu_{k+1}$ by \eqref{updateC-empir};
		\IF{$r_k\geq\beta_0$}
		\STATE $\theta_{k+1}=\tilde{\theta}_{k+1}$ and $B_{\theta_{k+1}}=T_{\tilde{\theta}_{k+1}}$;  
		\ELSIF{$|S_{k}|< N_{\max}$}
		\STATE $\theta_{k+1}=\theta_{k}$ and $B_{\theta_{k+1}}=\mathrm{SAMPLE}(\theta_{k+1},N)\cup B_{\theta_k}$;
		\STATE   $H=H\cup\{ (\tilde{\theta}_{k+1},T_{\tilde{\theta}_{k+1}})\}$; 
		\ELSE
			\STATE update $(\theta_{k+1},B_{\theta_{k+1}})$ by \eqref{forceaccept}; $H=H\setminus\{(\theta_{k+1},B_{\theta_{k+1}})\}$;
\ENDIF
 \STATE update $\phi_k$ by \eqref{critic-prob};
		 \STATE $k=k+1$;	
			\ENDWHILE
\end{algorithmic}
	\end{algorithm}

\section{Experiments}\label{Experiments}
To investigate the effectiveness and robustness of our stochastic trust region method for policy optimization among the state-of-the-art deep reinforcement learning algorithms, we make a comprehensive comparison using OpenAI's Baselines \cite{baselines} and Spinningup\footnote{An educational resource produced by OpenAI, \url{https://spinningup.openai.com}. }. The discrepancy between the implementations in these two repositories are primarily in the preprocessing of the environments and the structure of the networks. In Baselines, some additional wrappers are included for normalizing the environment informations and the size of the networks varies for the algorithms, while in Spinningup the observations and the reward function are taken from the environments directly and the acquiescent network architecture is shared among the methods. As an on-policy algorithm, we compare our method with TRPO and PPO on a range of continuous control tasks and Atari game playings using the benchmark suite OpenAI Gym. In order to facilitate comparisons, we take the distance function in our method as the KL divergence since it is differentiable and is commonly used. We should point out that  there is a gap between the implemented PPO in Baselines and the theoretical algorithm in \cite{schulman2017proximal}. Ilyas et al. \cite{ilyas2018deep} state that,  the optimizations which are not part of the core algorithm develop the practical success of PPO to a large extent and the theoretical framework might be negligible in practice. To study the learning capability of the policy-based and values-based methods, we take a comparison with DDPG, TD3, and soft actor-critic (SAC) method \cite{haarnoja2018soft} which is an energy-based off-policy algorithm. They are known to perform well in continuous controls at the expense of large interactions and long training time. Since these methods are incapable to handle discrete problems directly, we only compare with them in continuous tasks.

\subsection{Continuous Controls}
We test nine representative robotic locomotion experiments using the simulator MuJoCo\footnote{\url{http://www.mujoco.org}} in Gym. The problem is simulating a robot to win highest returns with fluent and safe movements. The unknown dynamics, non-smooth reward shape and the high dimensionality make the problems being challenging.

For these continuous tasks, we use the Gaussian distribution to characterize the conditional probability $\pi_{\theta}(a|s)$. As we described in section \ref{Separating},  the policy is defined by the normal distribution $N(\mu(s;\theta^\mu), \mathrm{diag}(\exp(2\theta^{\sigma})))$. The mean function $\mu(\cdot;\theta^\mu)$ is parameterized by a neural network with $tanh$ units. Moreover, for the value network where the parameter is denoted as $\phi$, we use the same architecture as the mean function except that the dimension of the last layer is one. We update the parameter $\theta=[\theta^\mu,\theta^\sigma]$ by the alternative models in \eqref{sep-gaussian} and \eqref{sigma-general}, and train the value network using the Adam \cite{kingma2014adam} method simultaneously.

During our experiments, we take $N=2048$ for each simulation and set the initial trust region coefficient $\mu_0=0.05$, the adjustment factors $\gamma_1=2$, $\gamma_2=0.8$, $\gamma_3 = 0.6$, $\mu_{\max}=0.1$ and $\mu_{\min}=0.01$. Empirically, the trial point is accepted once $\eta(\tilde{\theta}_{k+1})>\eta(\theta_k)$, i.e., $\beta_1=0$. The compulsive acceptance takes place after four consecutive rejections. FIG \ref{fig:compare} presents the training curves over these environments in Baselines. Each task is run for one million time steps over five random seeds of the network initialization and Gym simulator. The horizontal and vertical axes are training time steps and the empirical total expected rewards, respectively. The solid curves represent the mean values of five independent simulations and the shaded areas correspond to the standard deviation. TABLE \ref{maxcompare} summaries the maximal averaged return and a single standard deviation over five trials. A higher maximal average return indicates that the algorithm has ability to capture the better agent. We find that our method outperforms or matches TRPO and PPO in almost the same amount of time over all experimented environments. 	
\begin{figure}[htb]
		\centering
		\subfloat[]{\label{fig:compare_1}\includegraphics[width=1.8in]{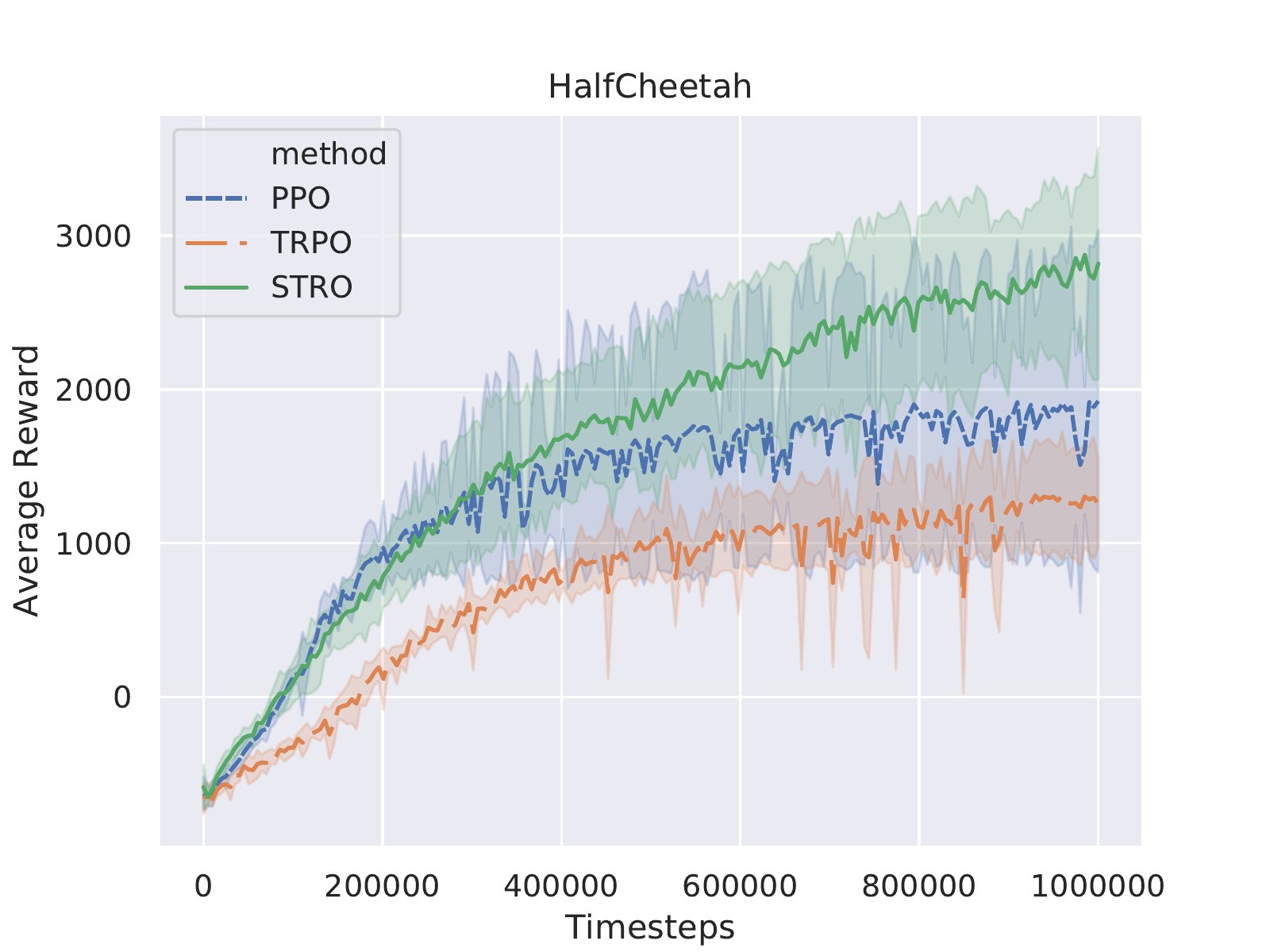}}
		\subfloat[]{\label{fig:compare_2}\includegraphics[width=1.8in]{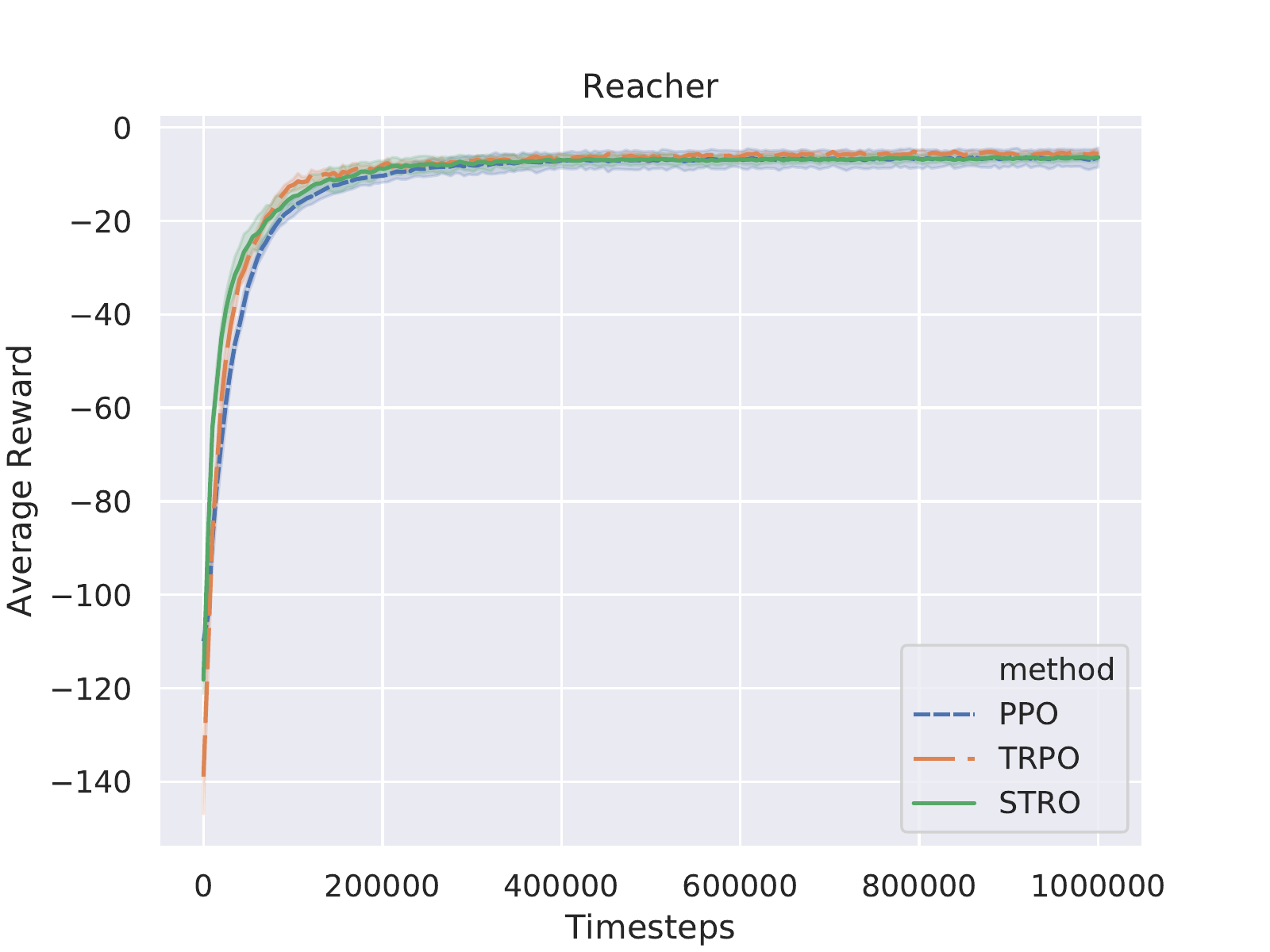}}
		\subfloat[]{\label{fig:compare_3}\includegraphics[width=1.8in]{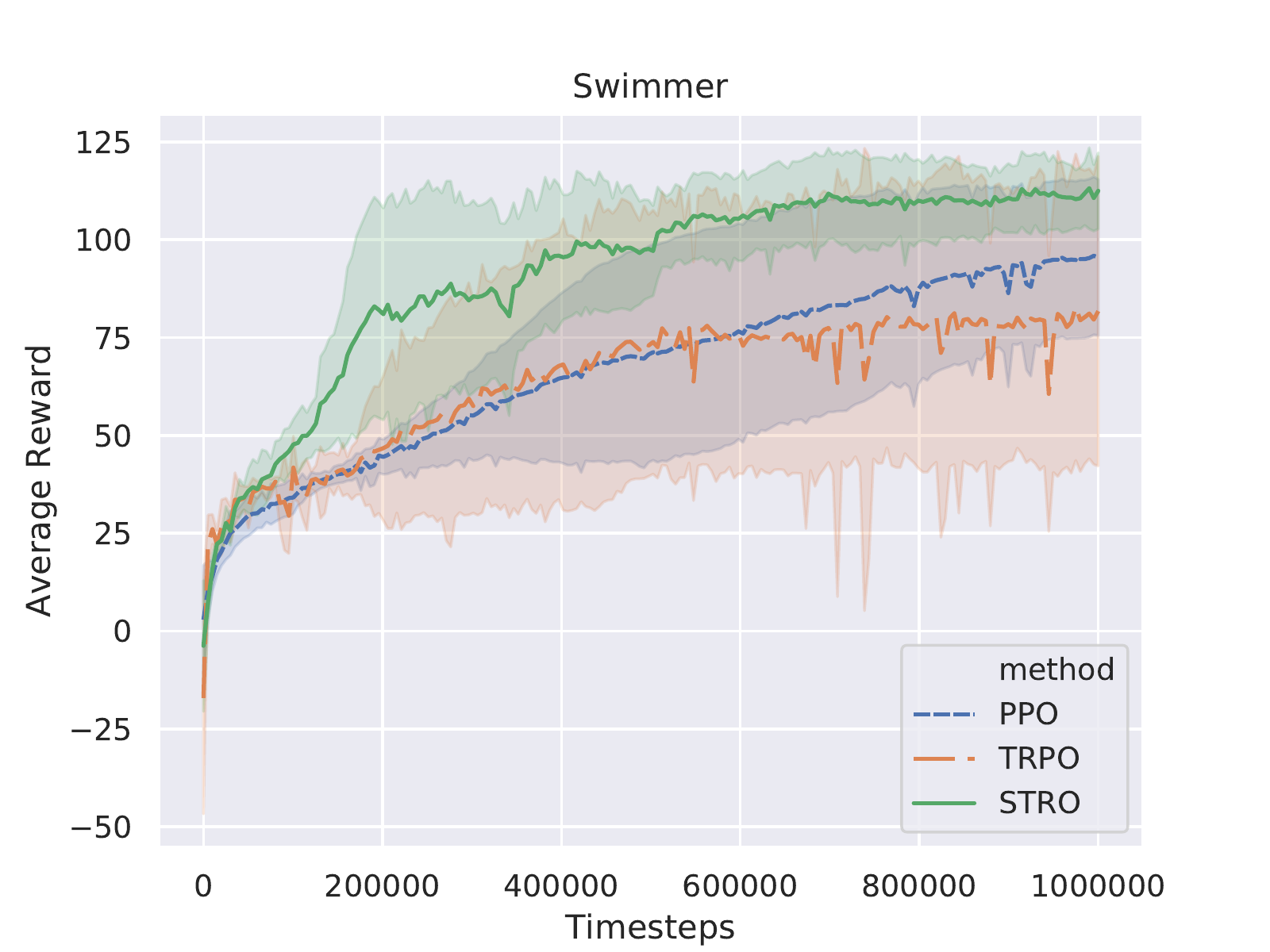}}
		\\
		\subfloat[]{\label{fig:compare_4}\includegraphics[width=1.8in]{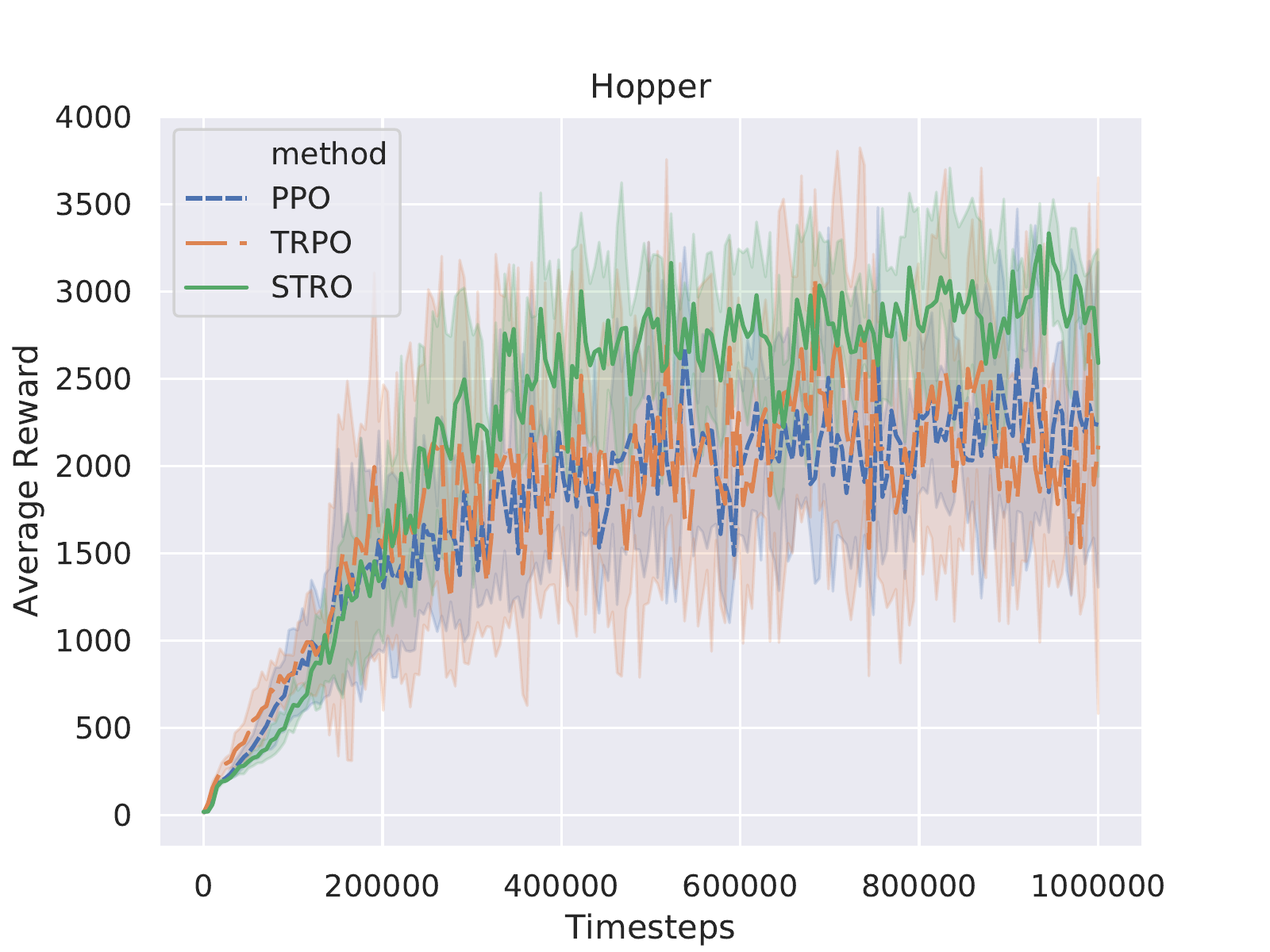}}
		\subfloat[]{\label{fig:compare_5}\includegraphics[width=1.8in]{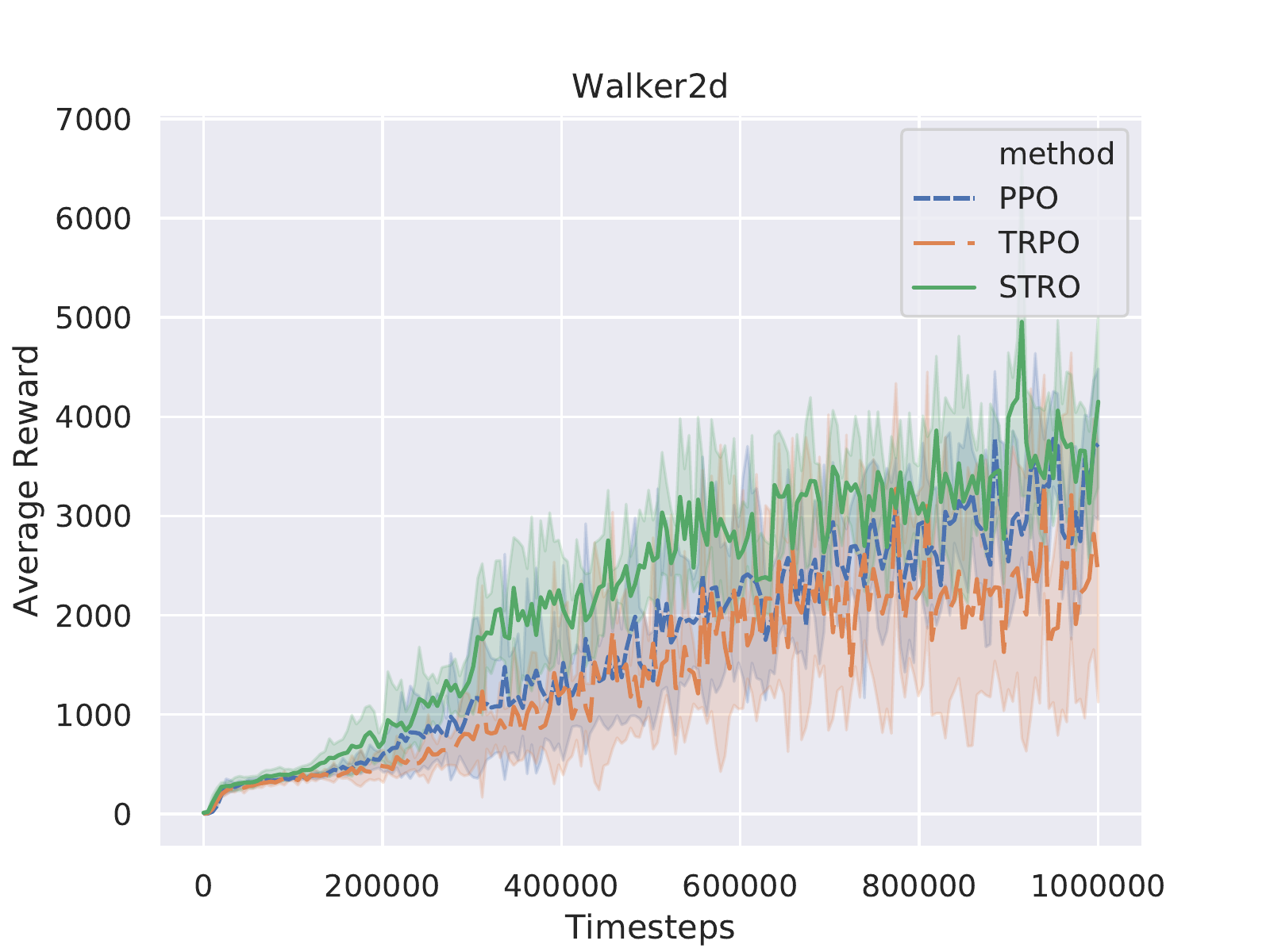}}
		\subfloat[]{\label{fig:compare_6}\includegraphics[width=1.8in]{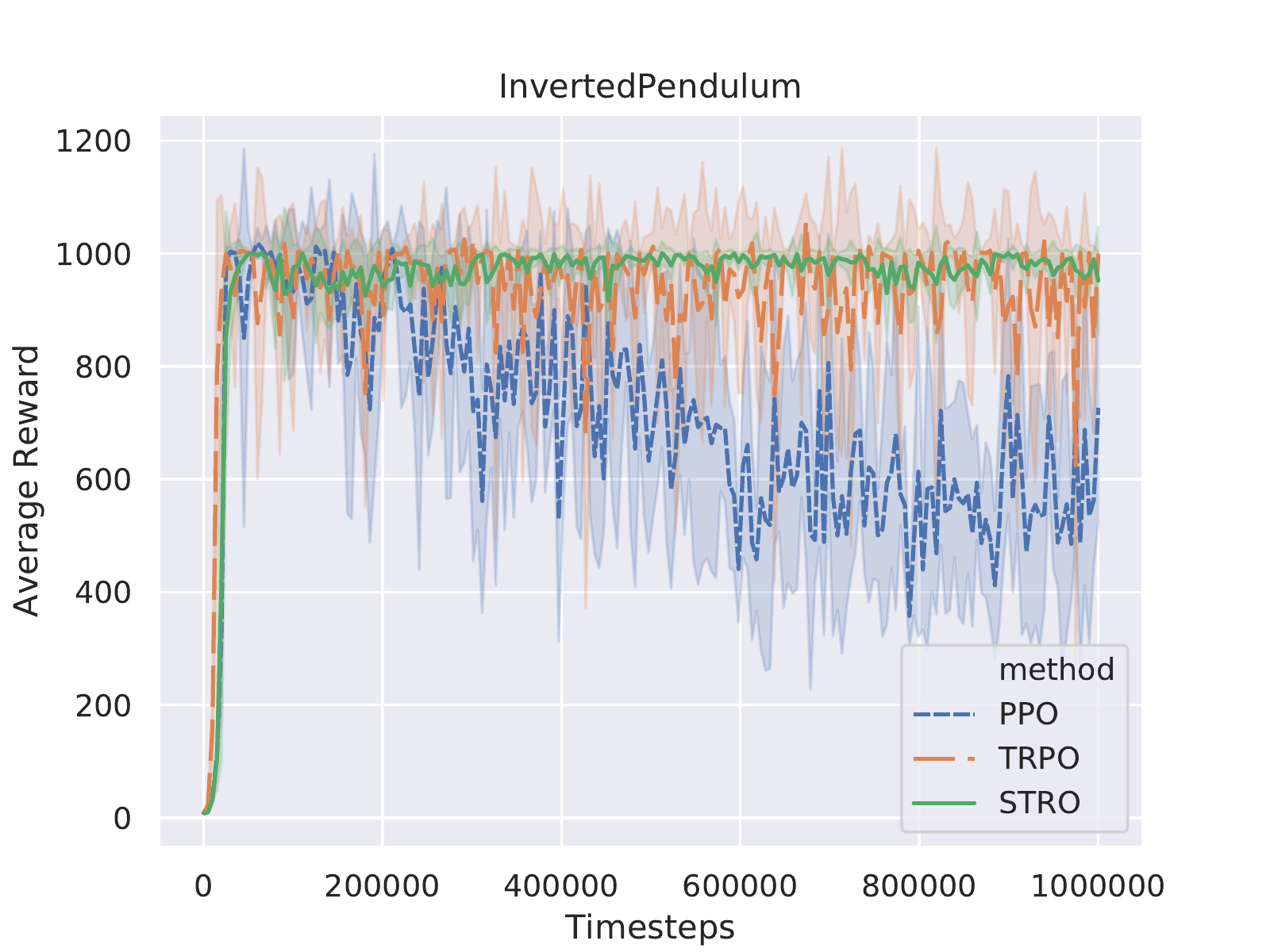}}
		\\
		\subfloat[]{\label{fig:compare_7}\includegraphics[width=1.8in]{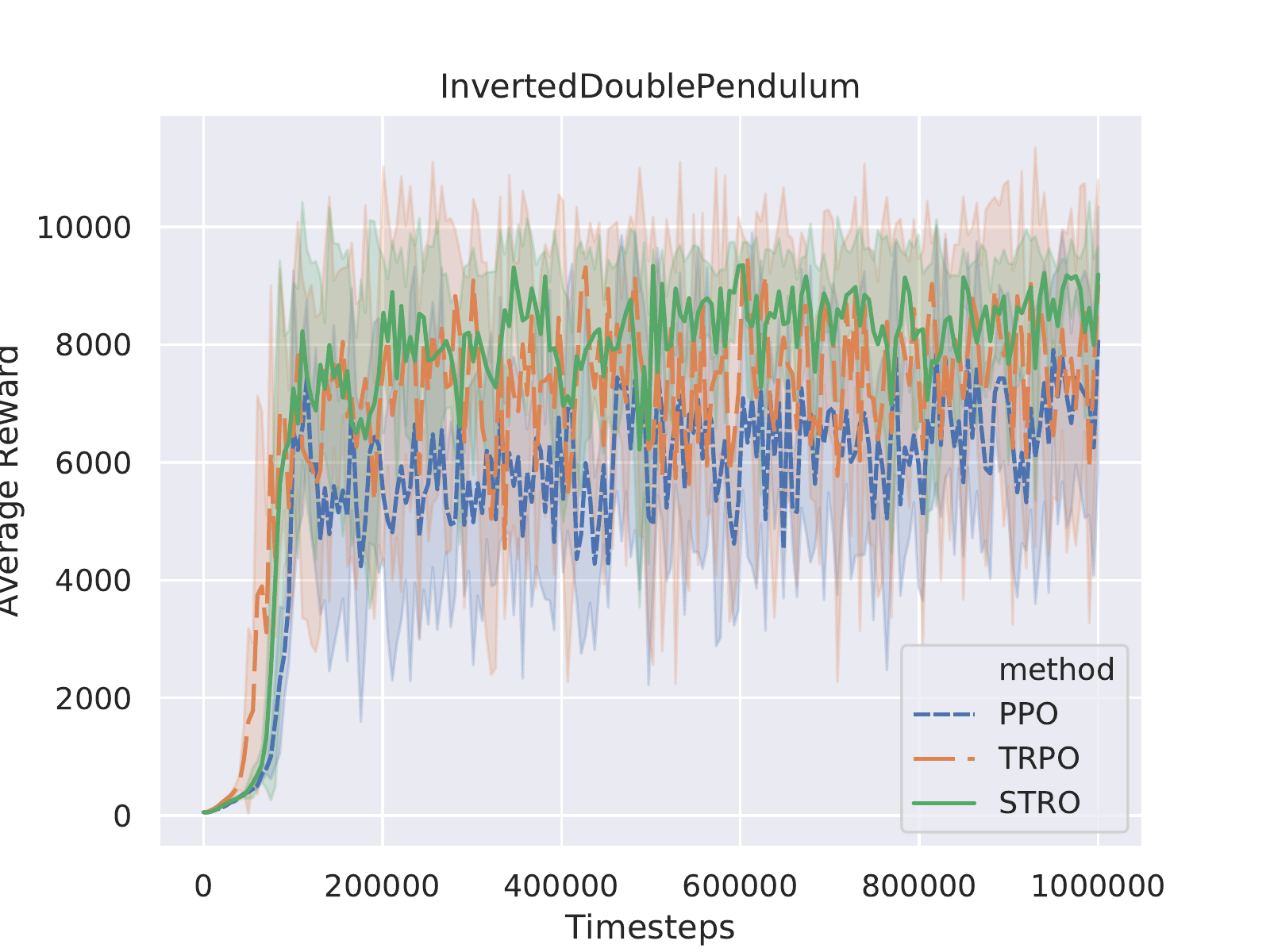}}
		\subfloat[]{\label{fig:compare_8}\includegraphics[width=1.8in]{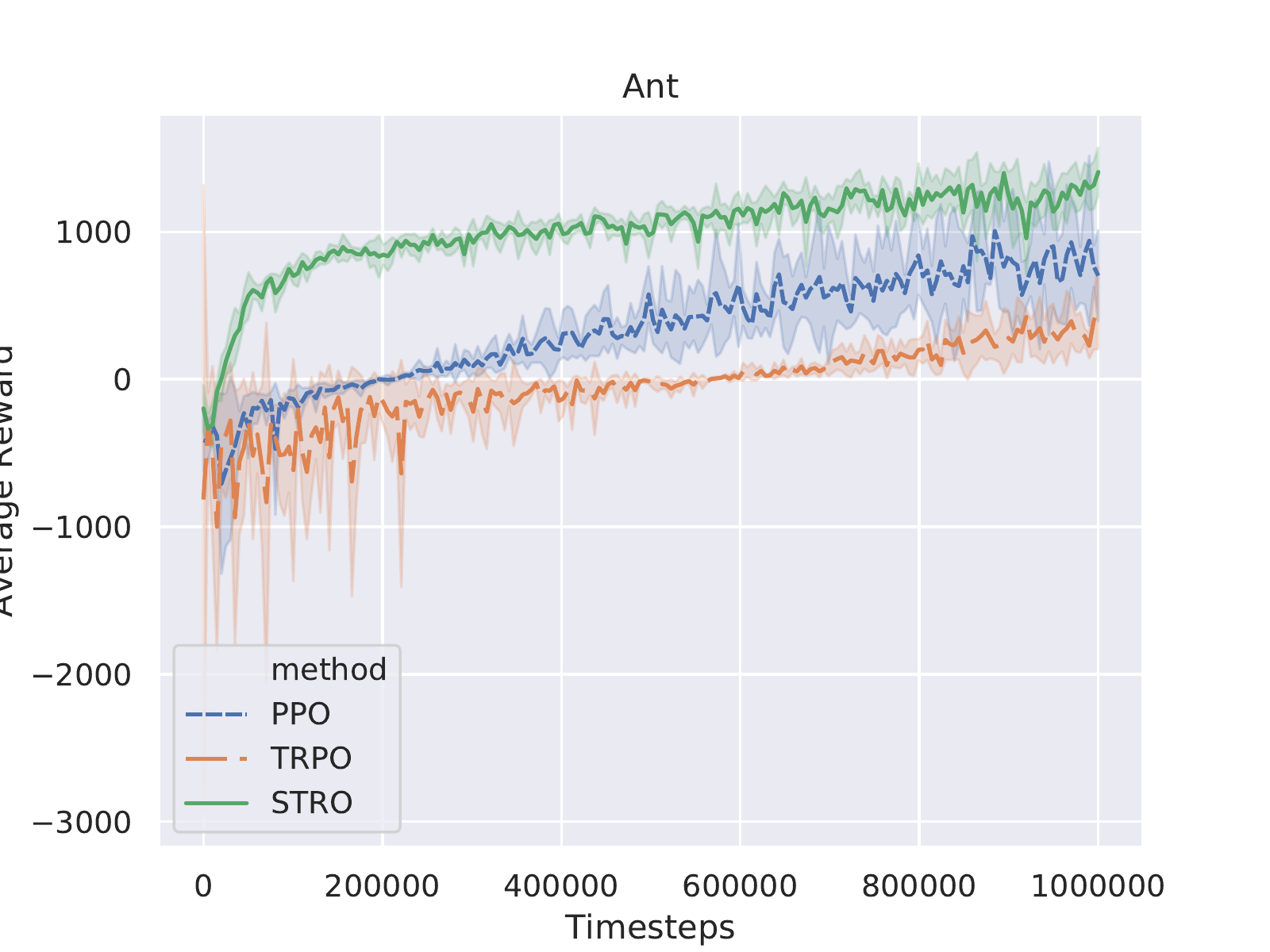}}
		\subfloat[]{\label{fig:compare_9}\includegraphics[width=1.8in]{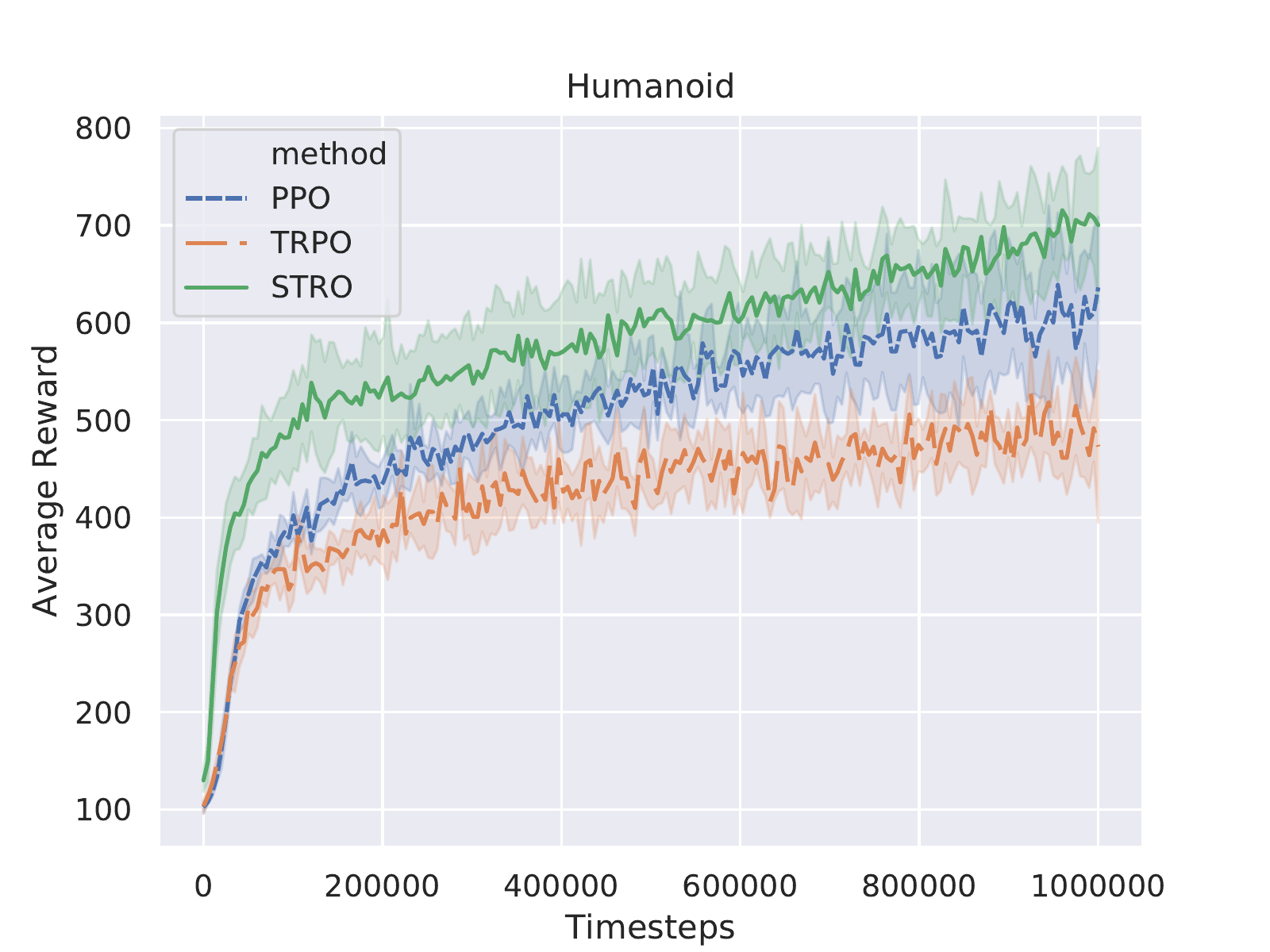}}\\
		\caption{Training curves on Mujoco-v2 continuous control benchmarks under Baselines.}
		\label{fig:compare}
	\end{figure}
\begin{table}
\caption{Max Average Reward $\pm$ standard deviation over 5 trials of 1e6 time steps under Baselines. The last column lists the average running time of these three methods.}
\centering
\begin{tabular}{|c|c|c|c|c|}
\hline
Env & PPO & TRPO & STRO & Time(min)\\
\hline
HalfCheetah & 1816$\pm$798 & 1406$\pm$390 & \textbf{2985}$\pm$491 & 36$\backslash$35$\backslash$54  \\
\hline
Reacher & -5$\pm$1 & -3$\pm$0 & -4$\pm$0 & 34$\backslash$38$\backslash$53  \\
\hline
Swimmer & 94$\pm$18 & 97$\pm$24 & \textbf{117}$\pm$9 & 36$\backslash$43$\backslash$46   \\
\hline
Hopper & 2475$\pm$91 & 3591$\pm$107 & 3463$\pm$81 & 35$\backslash$43$\backslash$48   \\
\hline
Walker2d & 3681$\pm$794 & 3935$\pm$936 & \textbf{5849}$\pm$2063 & 39$\backslash$43$\backslash$51   \\
\hline
InvPend & 1000$\pm$0 & 1000$\pm$0 & \textbf{1000}$\pm$0 & 41$\backslash$33$\backslash$49  \\
\hline
InvDoubPend & 9328$\pm$1 & 9340$\pm$7 & \textbf{9346}$\pm$3 & 43$\backslash$36$\backslash$49 \\
\hline
Ant & 793$\pm$206 & 666$\pm$79 & \textbf{1512}$\pm$101 & 43$\backslash$43$\backslash$46   \\
\hline
Humanoid & 679$\pm$64 & 627$\pm$11 & \textbf{784}$\pm$36 & 47$\backslash$44$\backslash$ 56  \\
\hline
\end{tabular}
\label{maxcompare}
\end{table}

As the recent studies \cite{henderson2018deep,ilyas2018deep} demonstrate that the experimental techniques in the environment wrappers, such as observation normalization, reward clipping, etc., have a dramatic effect on the numerical performance, we make another comparison in Spinningup where no extra modification is applied in the environments and the network structure is shared among all algorithms. From the results in FIG \ref{fig:spinup}, our algorithm still surpass TRPO and PPO on most experiments as in Baselines. Because the implementations are slightly different between Baselines and Spinningup, the values of the average reward in FIG \ref{fig:compare} and FIG \ref{fig:spinup} are different in some tasks, even for the same algorithm. Generally, as illustrated in TABLE \ref{maxcompare}, our method takes around 50 minutes in Baseines for each task over one random seed on average, slightly slower than PPO and TRPO. 
\begin{figure}[htb]
		\centering
		\subfloat[]{\label{fig:compare_10}\includegraphics[width=1.8in]{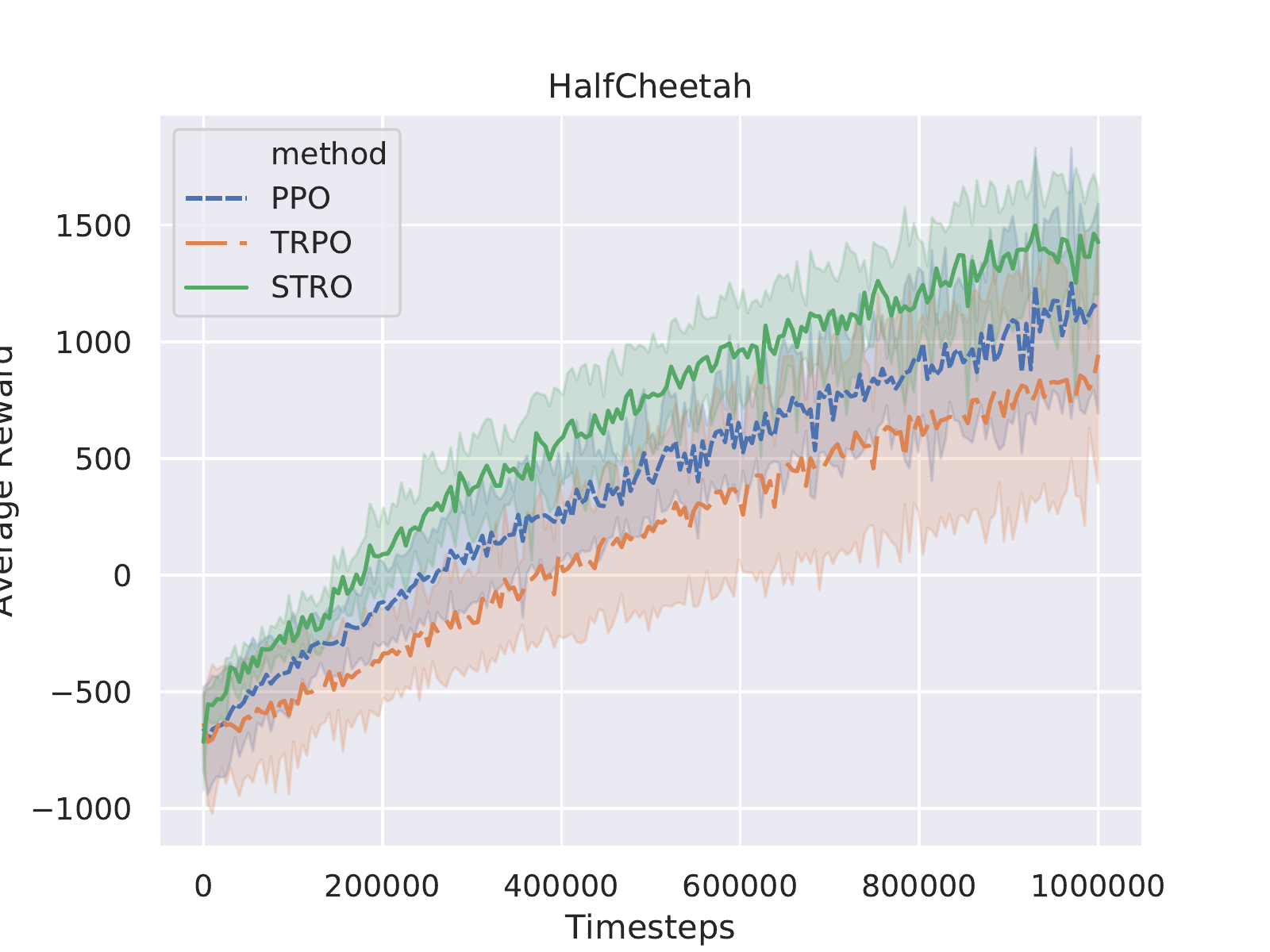}}
		\subfloat[]{\label{fig:compare_20}\includegraphics[width=1.8in]{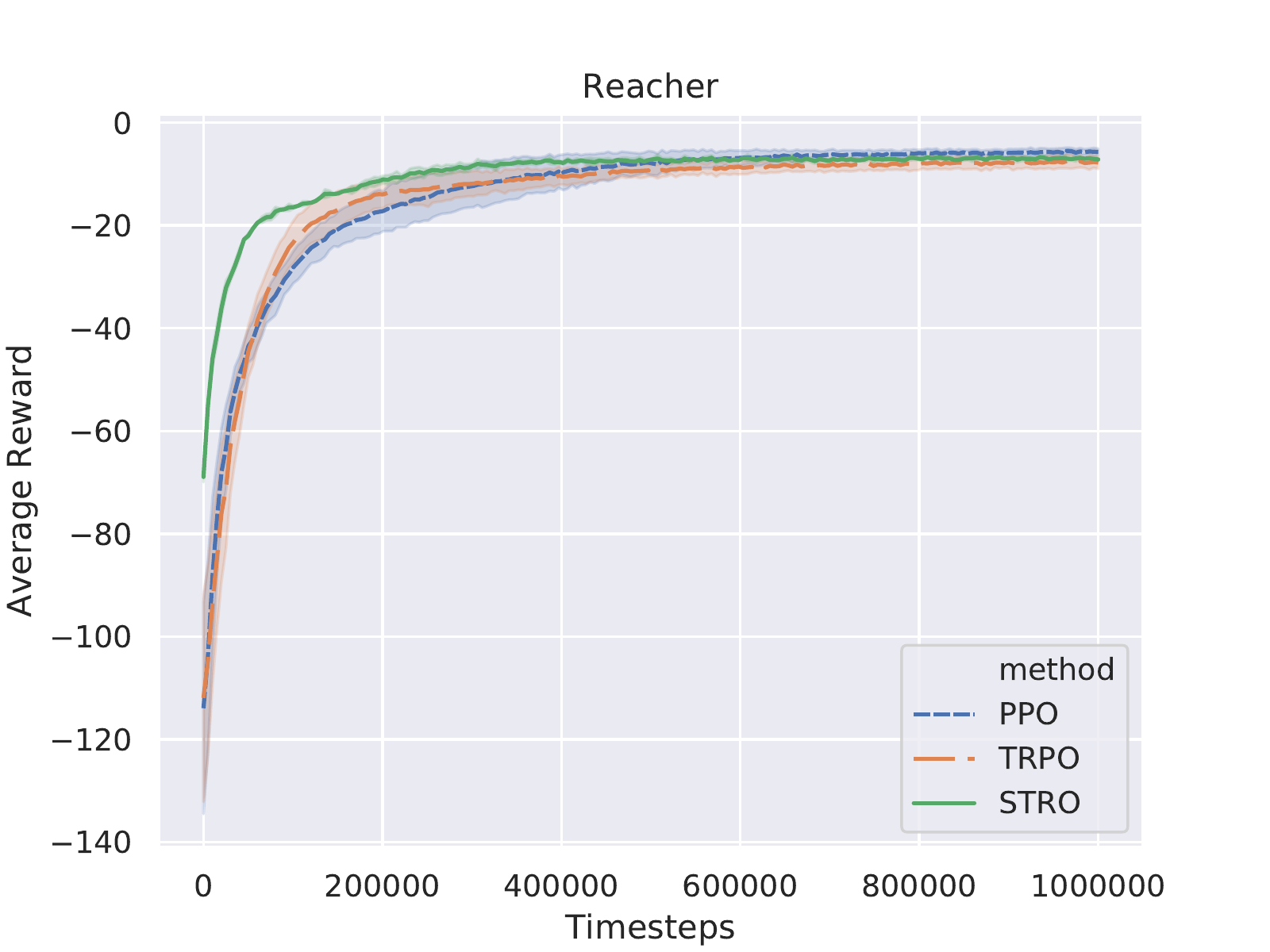}}
		\subfloat[]{\label{fig:compare_30}\includegraphics[width=1.8in]{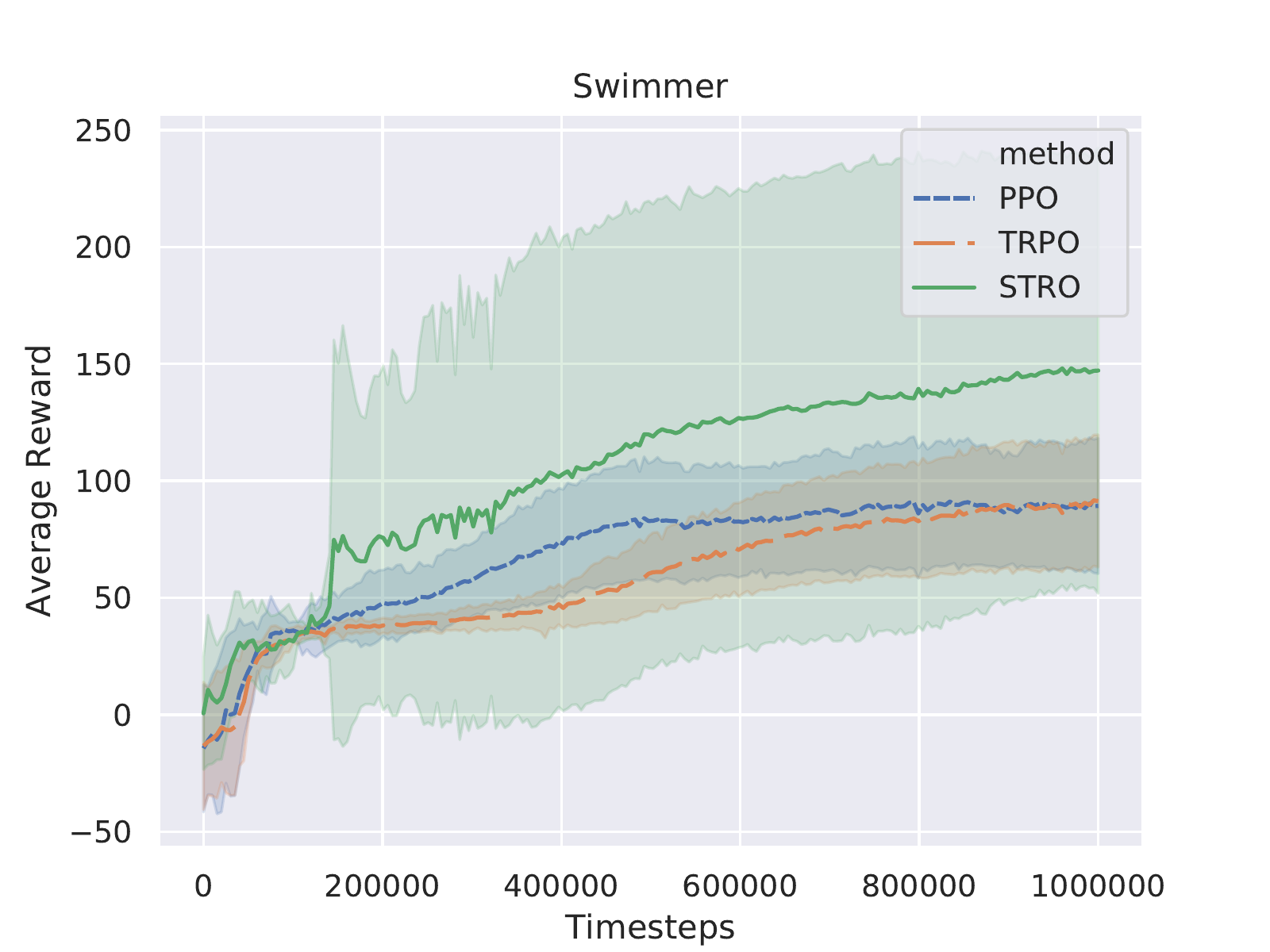}}
		\\
		\subfloat[]{\label{fig:compare_40}\includegraphics[width=1.8in]{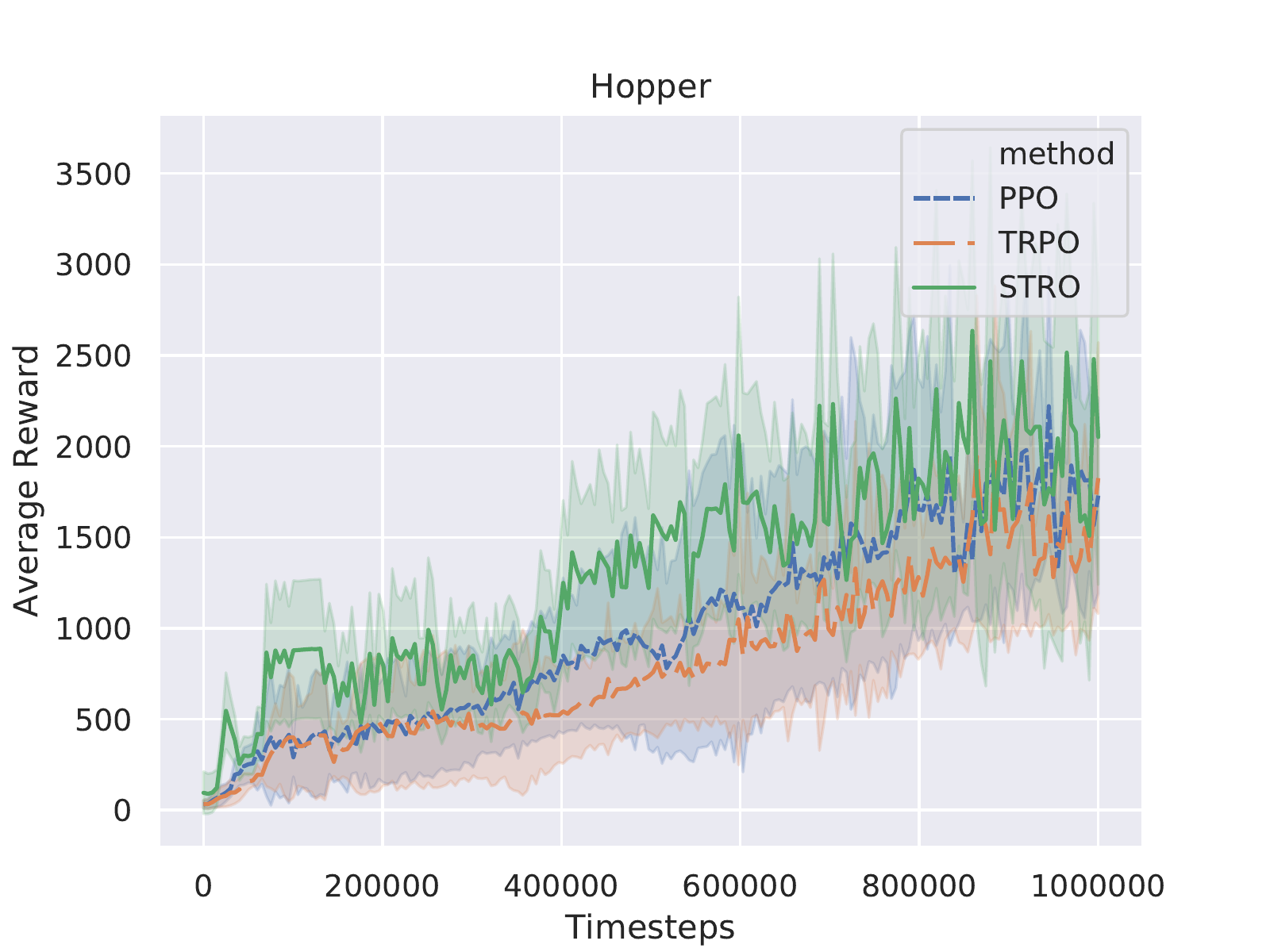}}
		\subfloat[]{\label{fig:compare_50}\includegraphics[width=1.8in]{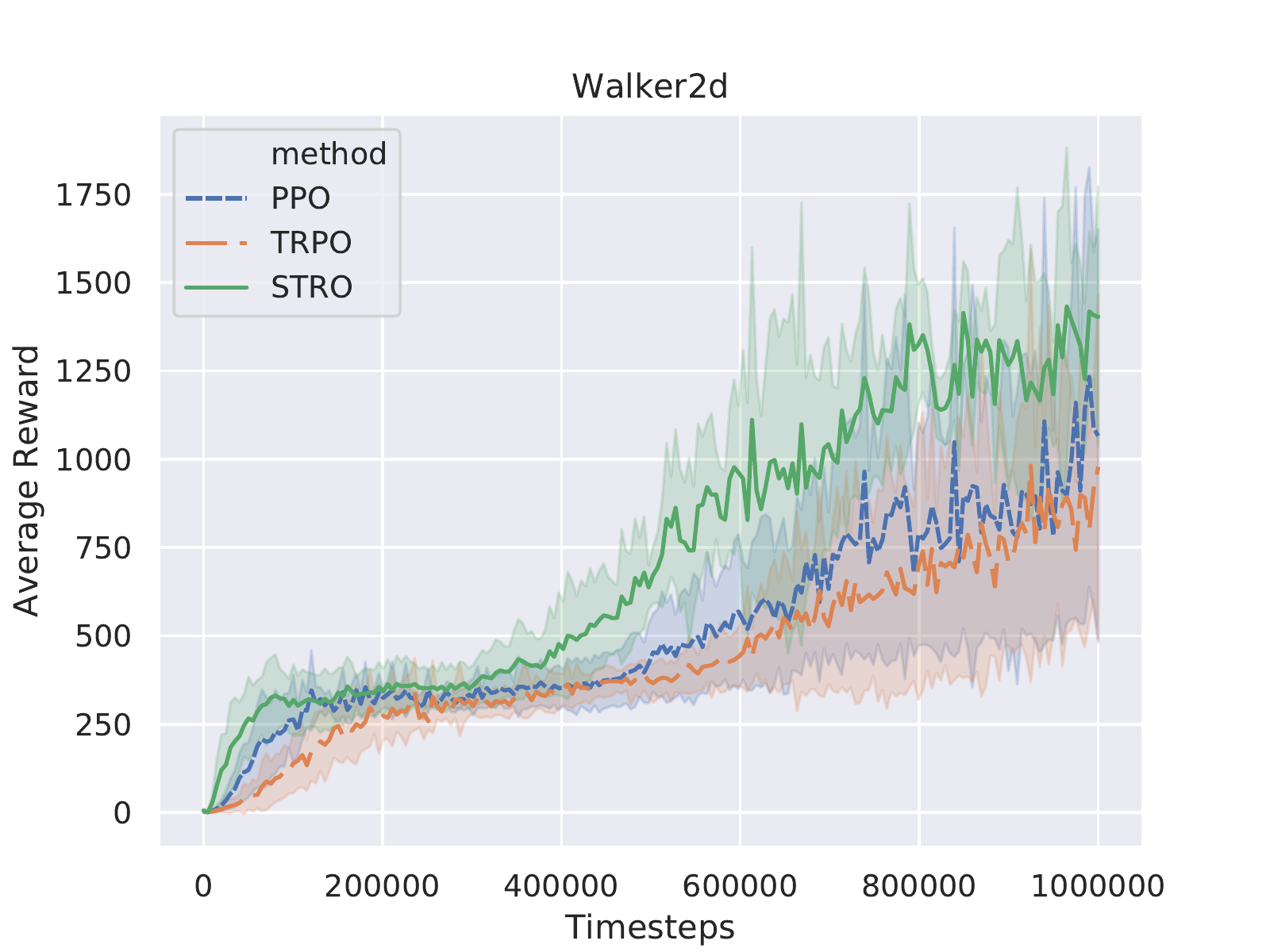}}
		\subfloat[]{\label{fig:compare_60}\includegraphics[width=1.8in]{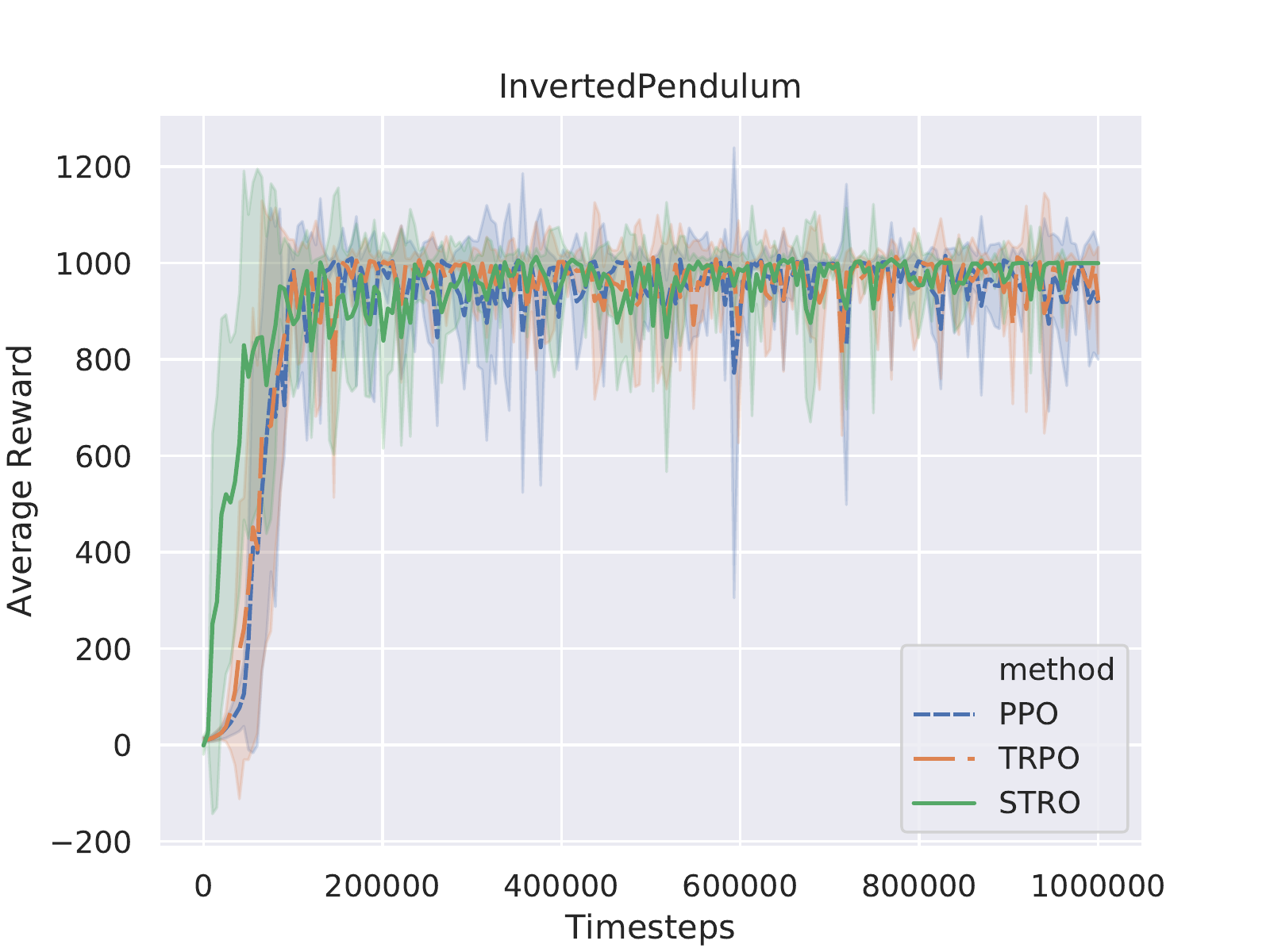}}
		\\
		\subfloat[]{\label{fig:compare_70}\includegraphics[width=1.8in]{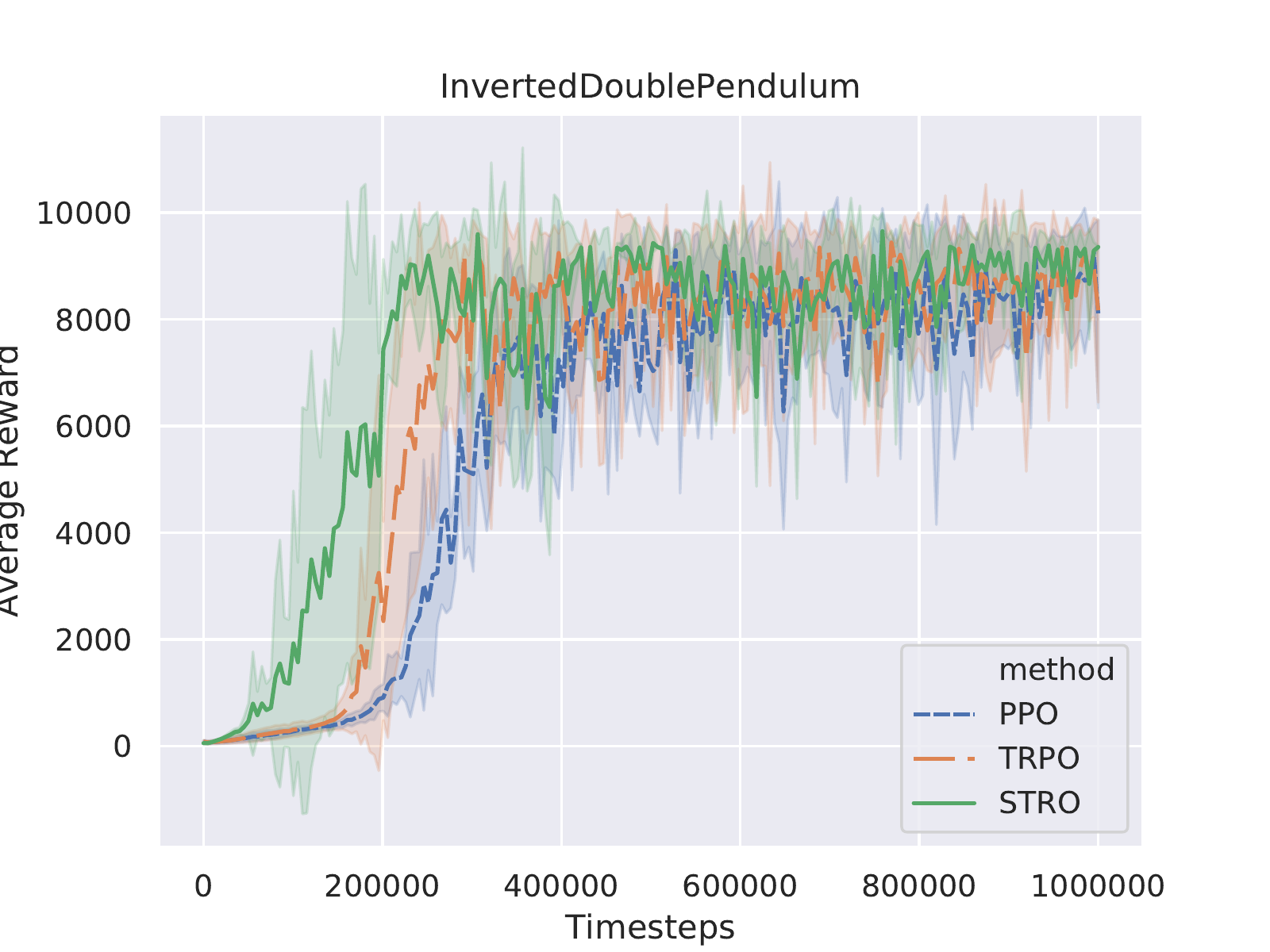}}
		\subfloat[]{\label{fig:compare_80}\includegraphics[width=1.8in]{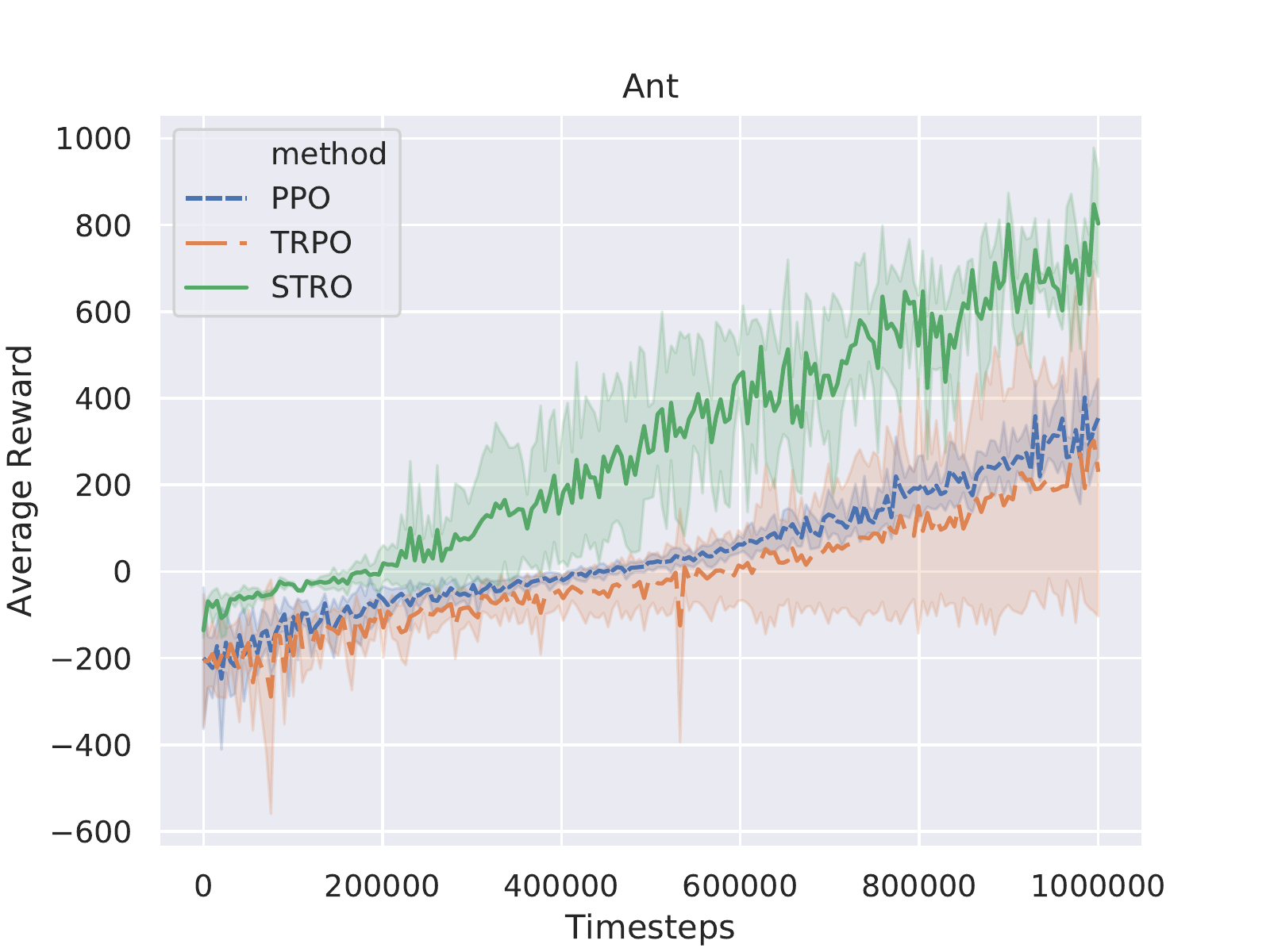}}
		\subfloat[]{\label{fig:compare_90}\includegraphics[width=1.8in]{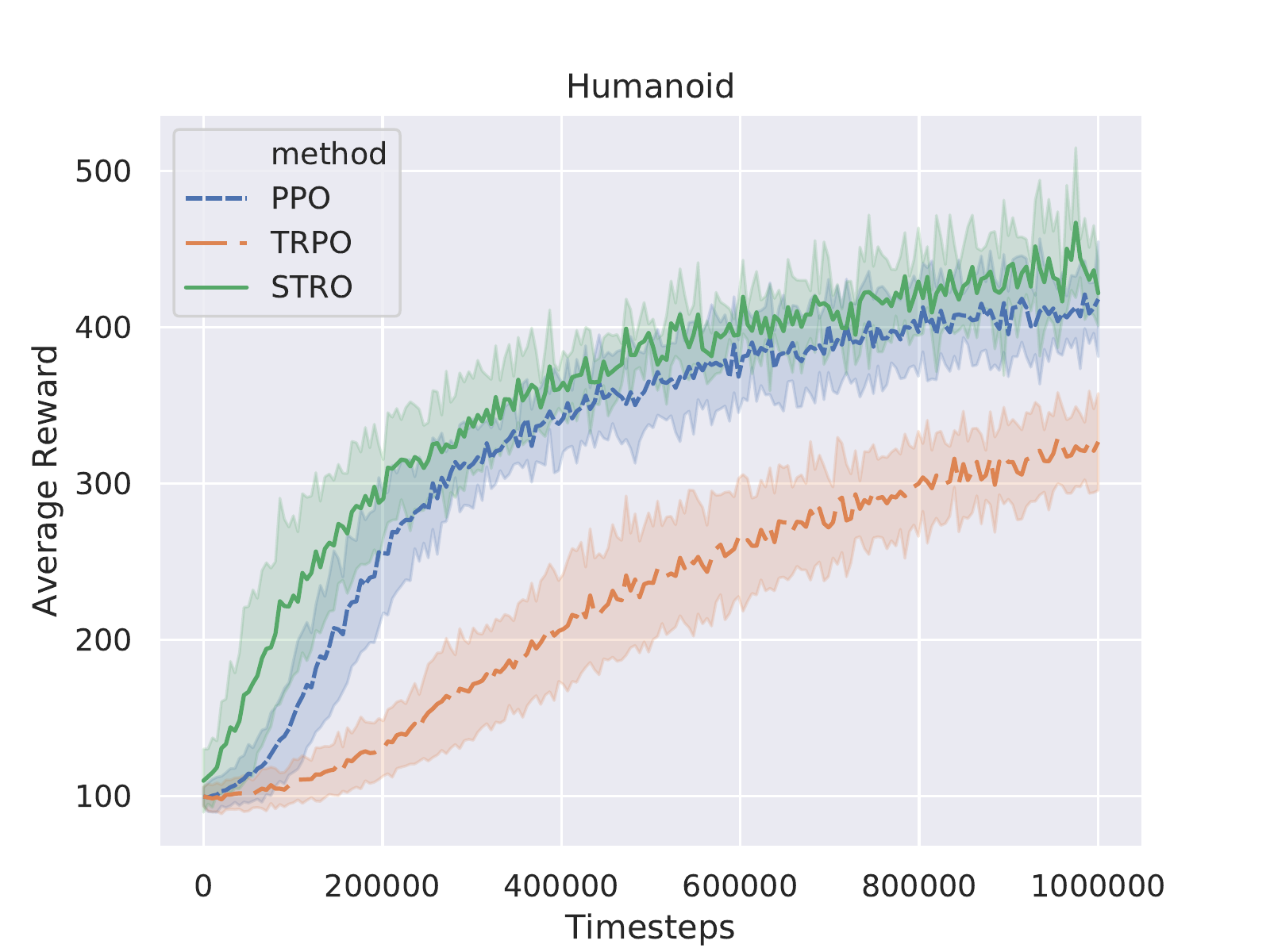}}\\
		\caption{Training curves on Mujoco-v2 continuous control benchmarks with Spinningup.}
	\label{fig:spinup}
	\end{figure}

We now compare with DDPG, TD3 and SAC. For these methods, in addition to the one million time steps for deterministic evaluation, i.e., no noise in sampling, they require another several million interactions for training which may take about several hours, and even longer in complex problems. In other words, the empirical convergence rate of these methods is much slower than that of policy-based methods. Due to the remarkable difference in the computational time among these algorithms, we compare their learning ability within a specified time. Since there is no standard implementation of TD3 and SAC in Baselines as a benchmark, we take the comparisons in Spinningup and terminate the algorithms after one hour training. We test these methods on nine robotic locomotion experiments. Specifically, the numerical results on $Swimmer$, $Ant$ and $Humanoid$ are reported in FIG \ref{fig:mix-1h}. Our method is always better than DDPG on most environments except in $HalfCheetah$. Moreover, in some problems our algorithm can surpass SAC and TD3, and in others it is slightly inferior than the best performed method. The comparisons demonstrate that our method is comparable with state-of-the-art deep reinforcement learning algorithms under such time constraint. Notably, for the other three methods, the simulations are deterministic which has inherent advantages compared to ours, sometimes even significant. 
\begin{figure}[htb]
		\centering
		\subfloat[]{\includegraphics[width=1.8in]{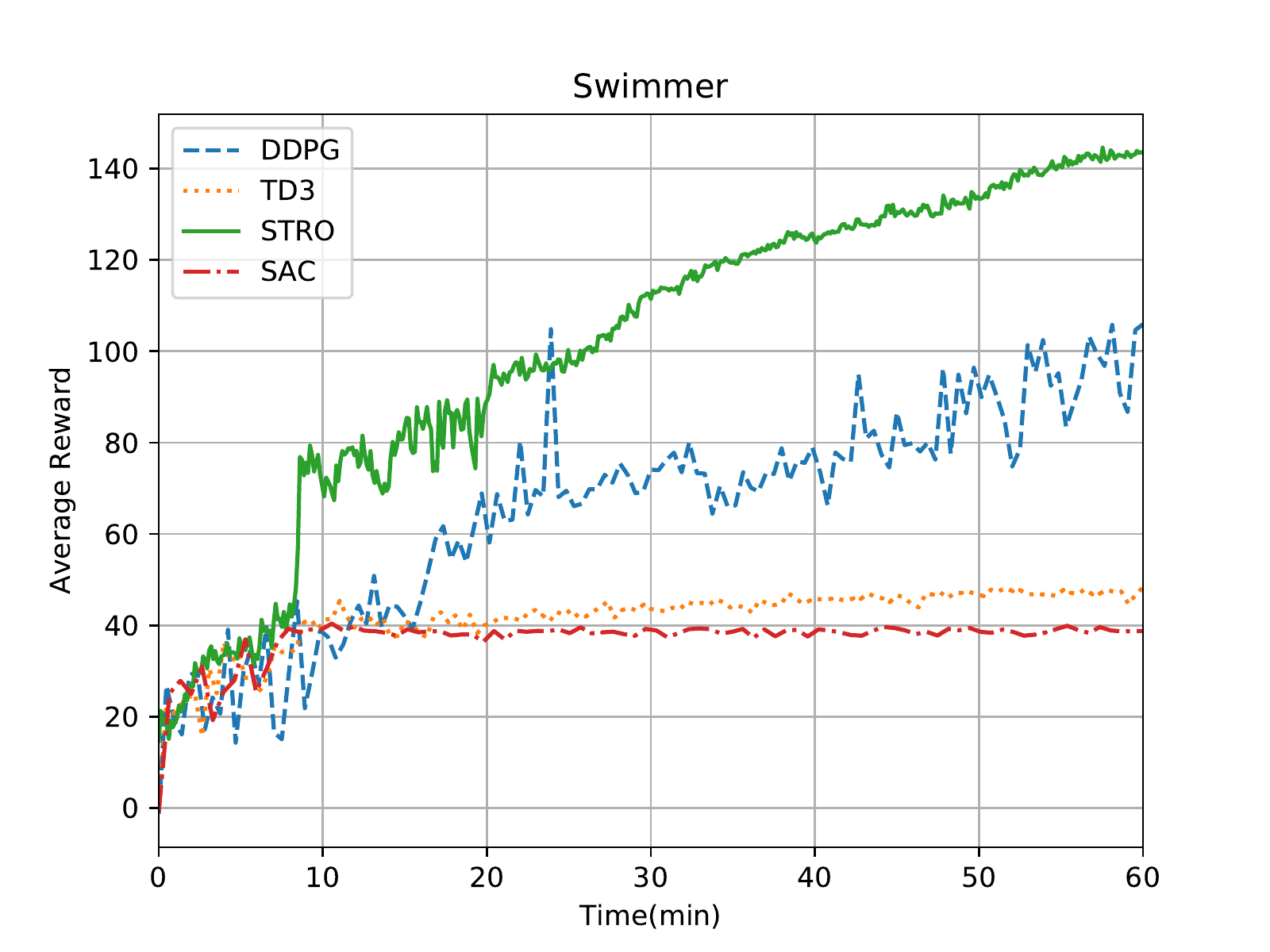}}
		\subfloat[]{\includegraphics[width=1.8in]{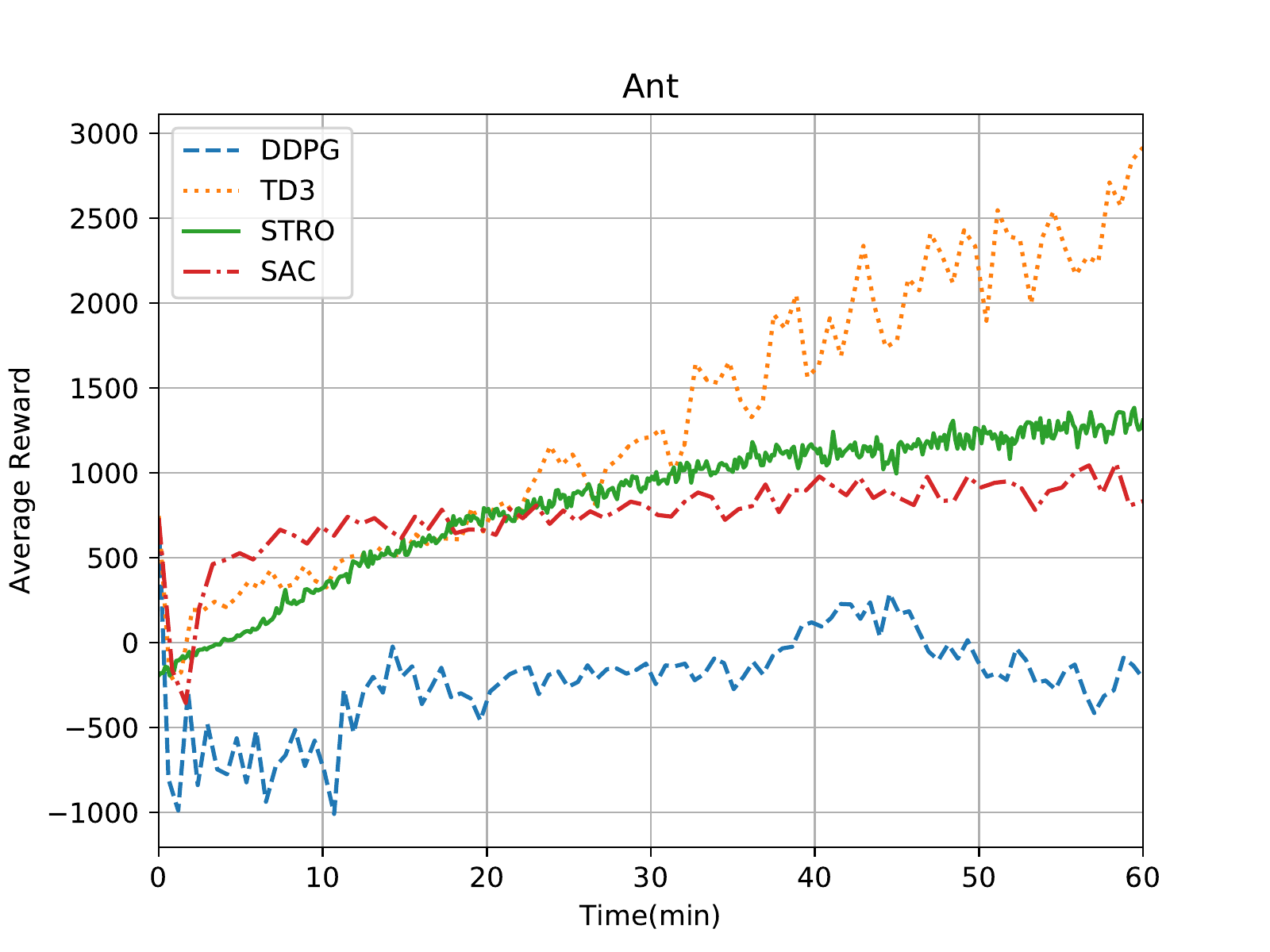}}
		\subfloat[]{\includegraphics[width=1.8in]{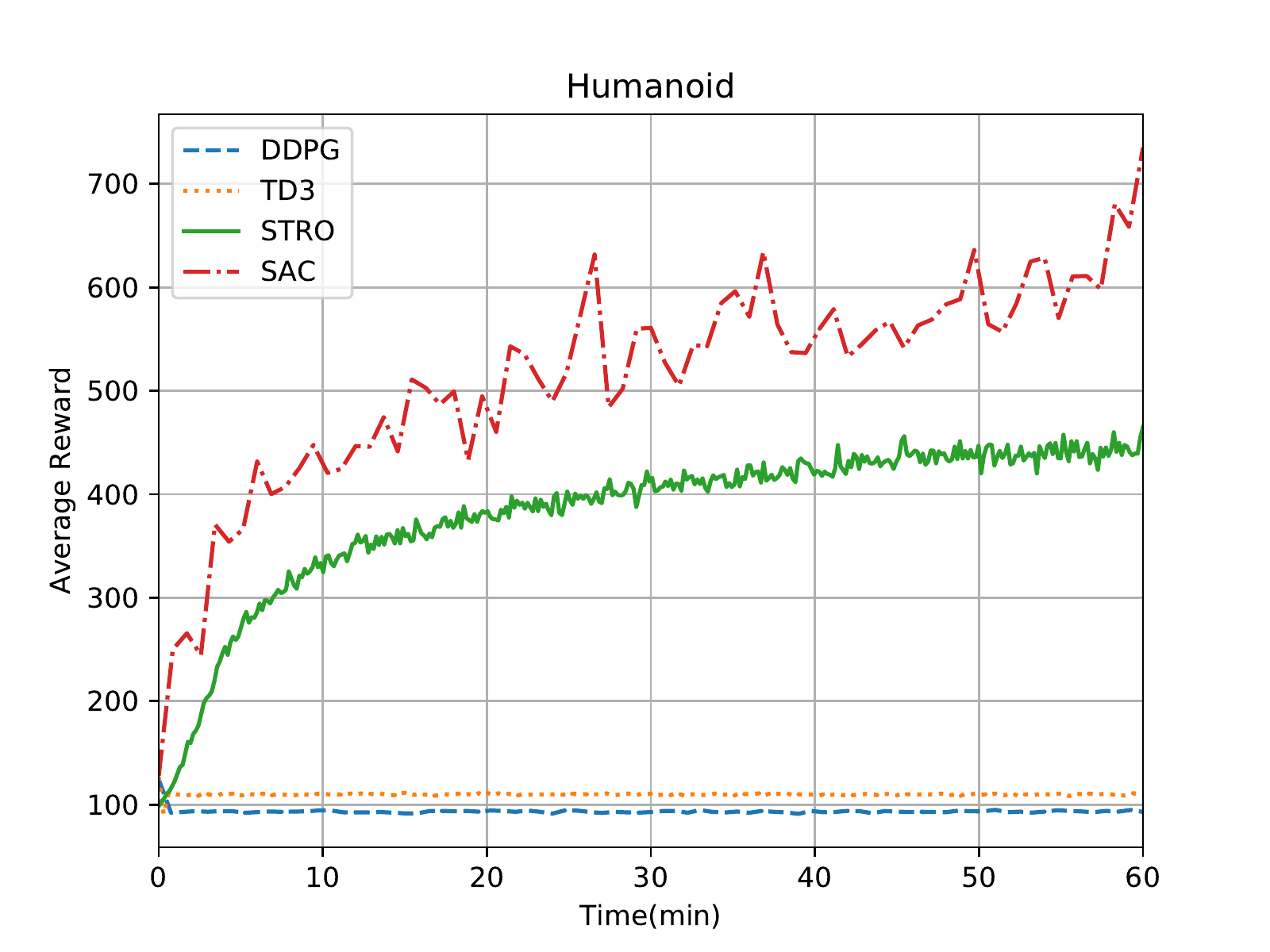}}
		\caption{One hour training curves on MuJoCo-v2 continuous control benchmarks under SpinningUp.}
		\label{fig:mix-1h}
	\end{figure}

\subsection{Discrete Control}
To evaluate our method on discrete problems, we randomly test 9 Atari games which have partially observed states in the form of images and discrete action spaces. The complex observations, high dimensionality, delayed reward and many other challenging elements make them to be extremely tough to learn. Different from the robotic locomotion experiments, the game images are preprocessed and feed into a convolutional neural network with $softmax$ operator in the last layer to represent a categorical policy. The value network has a similar structure as the policy network except the dimension of the output. Our method is implemented in Baselines, and the network architecture is the same as the one in PPO. For all tested games, we take $N=2048$ for each simulation and initialize the trust region coefficient $\mu_0=0.01$. We set $\mu_{\max}=0..5$ and $\mu_{\min}=0.005$. The maximal average return over 100 episodes on several games is reported in TABLE \ref{maxcompare-atari}, and the learning curves are plotted in FIG \ref{fig:atari} to illustrate the generalization of our method. Generally, our algorithm reaches higher or almost the same average reward in all tested environments.

\begin{figure}[htb]
		\centering
		\subfloat[]{\includegraphics[width=1.8in]{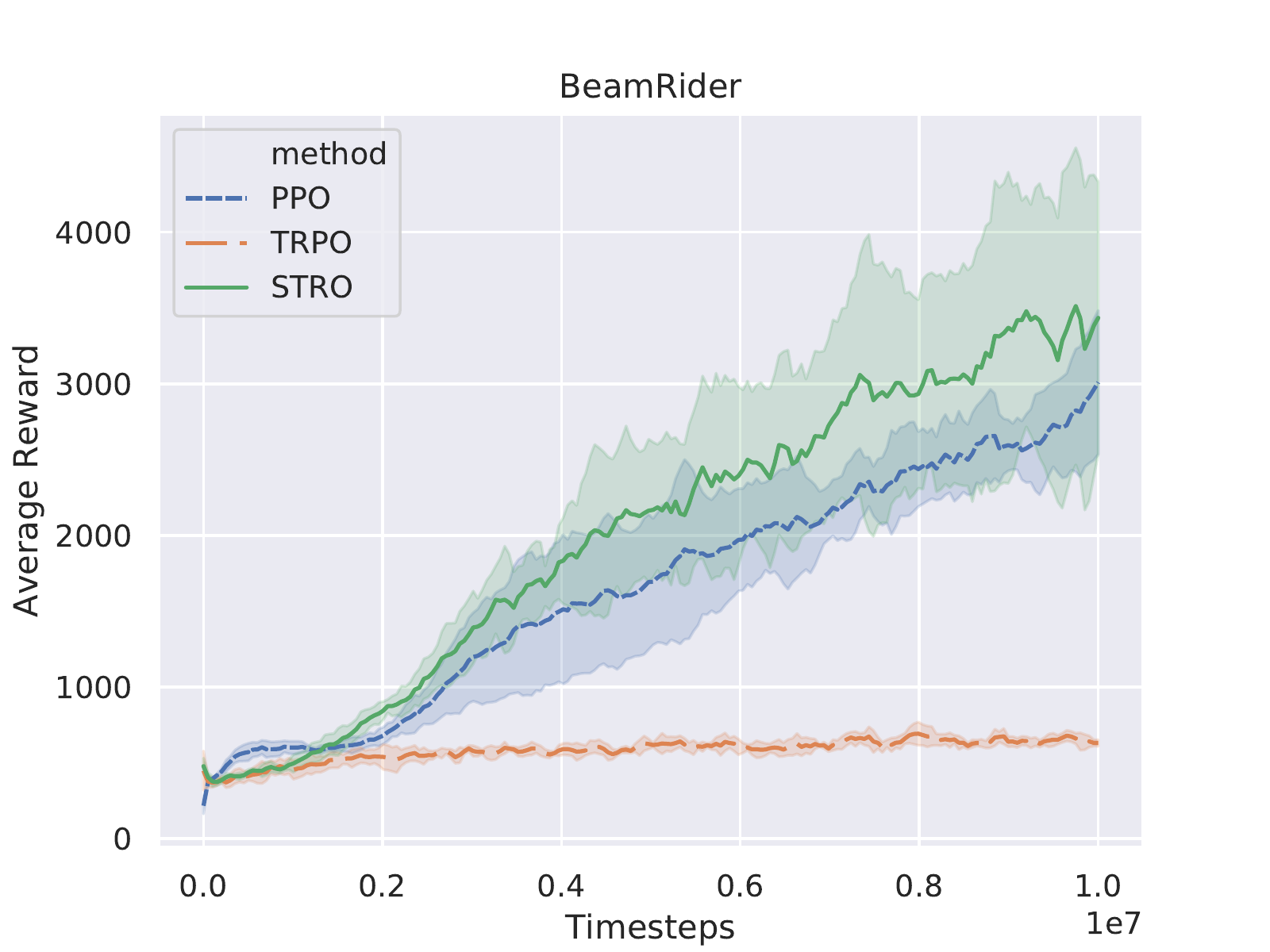}}
		\subfloat[]{\includegraphics[width=1.8in]{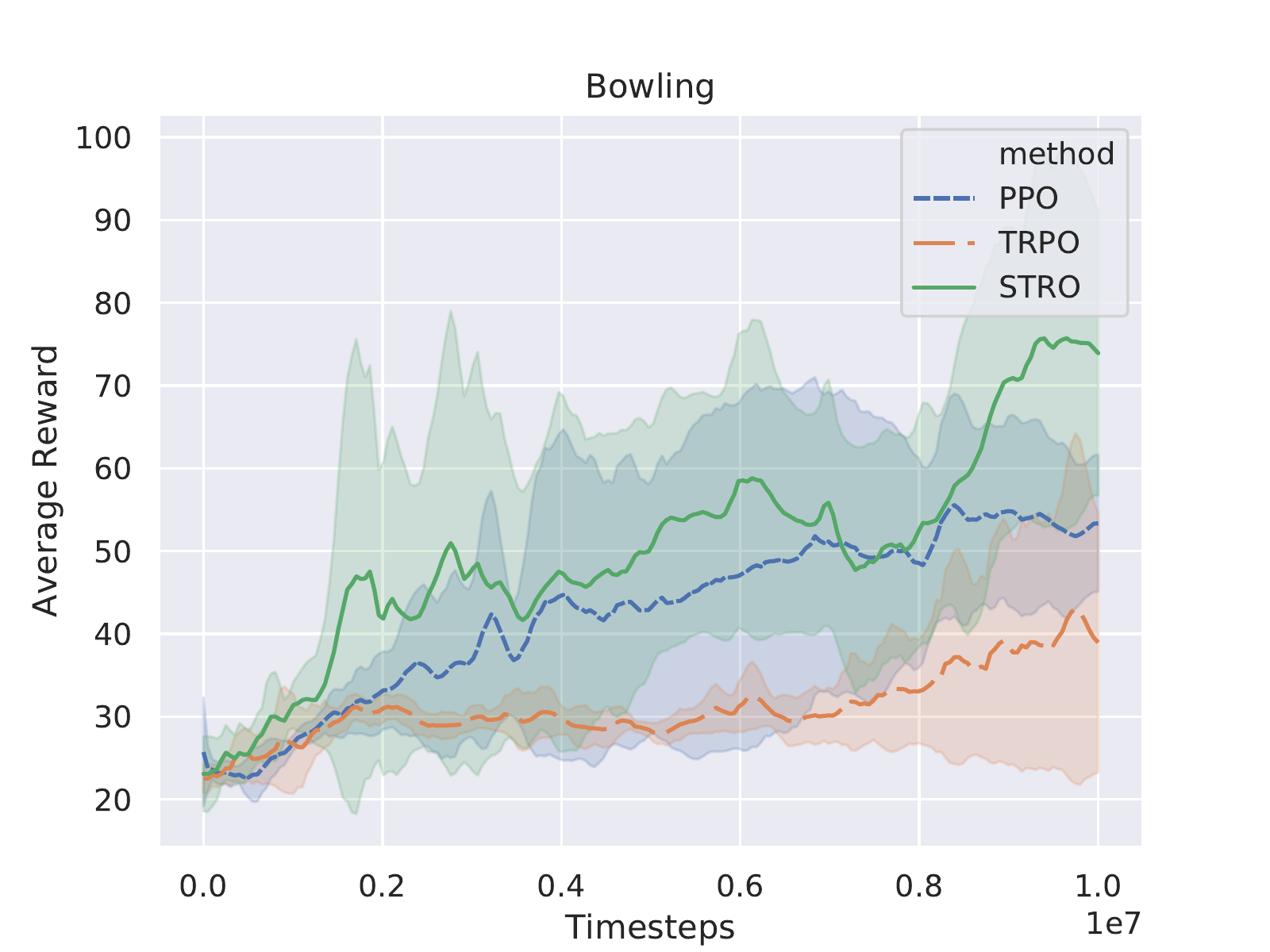}}
		\subfloat[]{\includegraphics[width=1.8in]{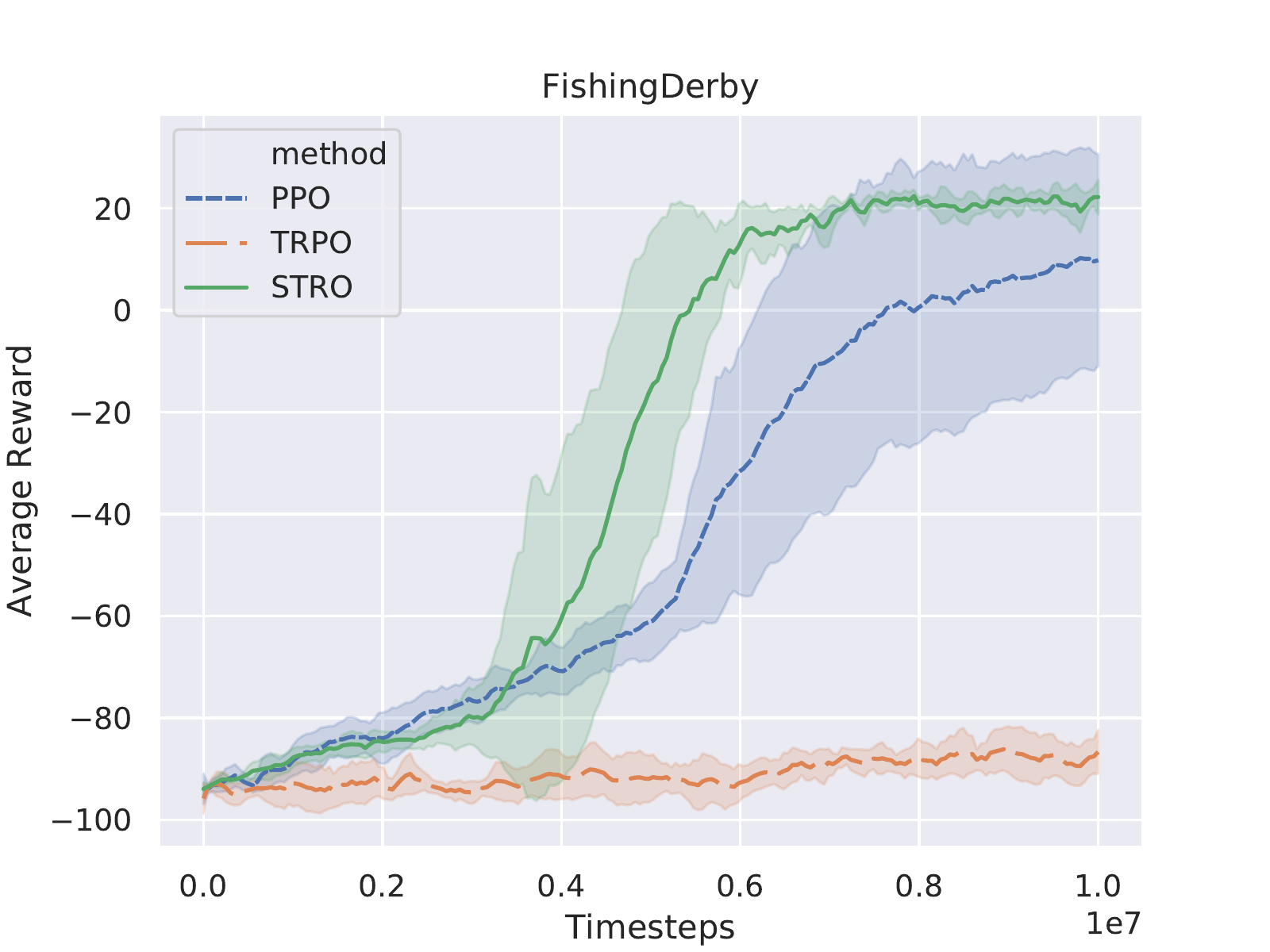}}
		\\
		\subfloat[]{\includegraphics[width=1.8in]{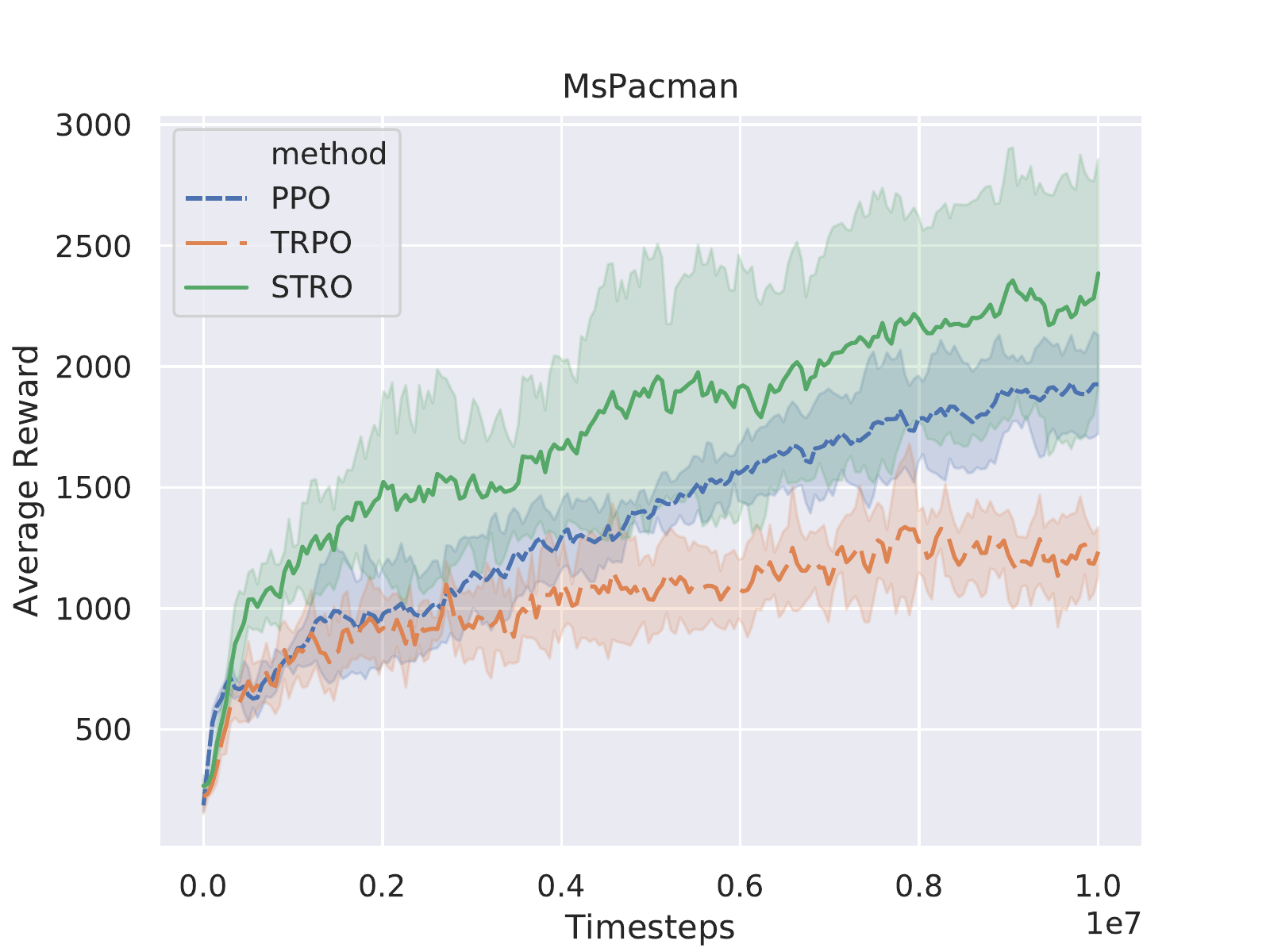}}
		\subfloat[]{\includegraphics[width=1.8in]{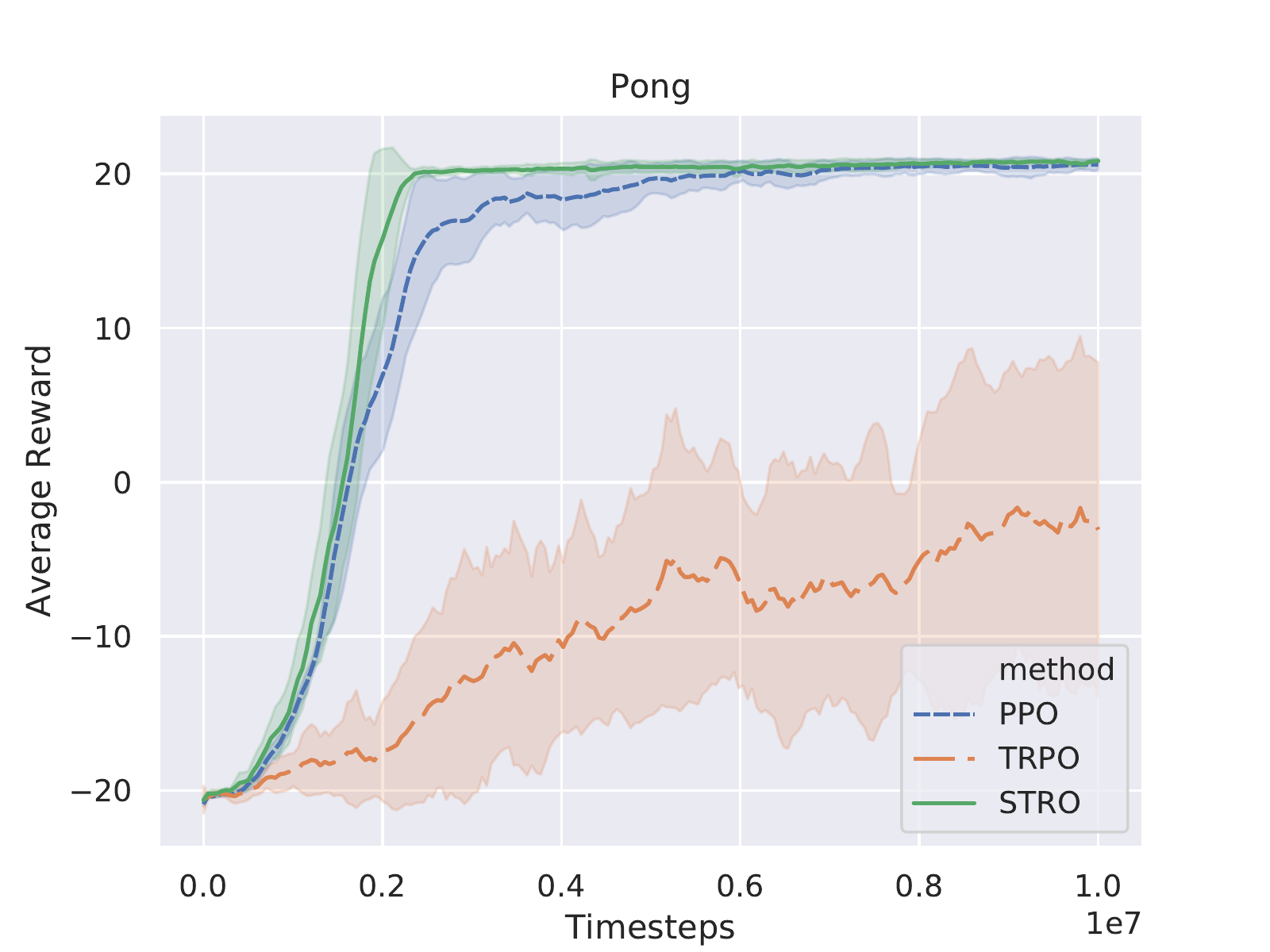}}
		\subfloat[]{\includegraphics[width=1.8in]{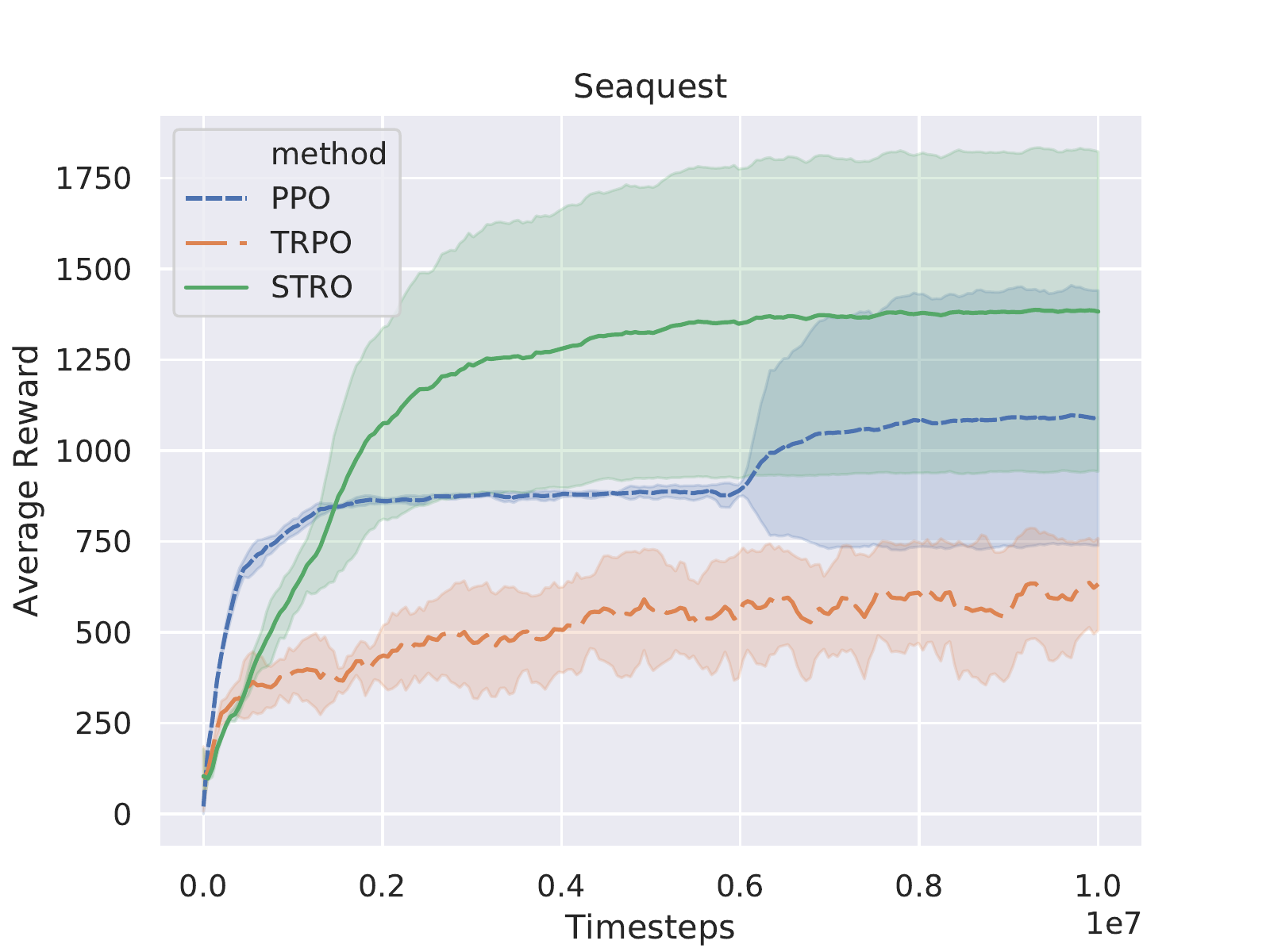}}
		\\
		\subfloat[]{\includegraphics[width=1.8in]{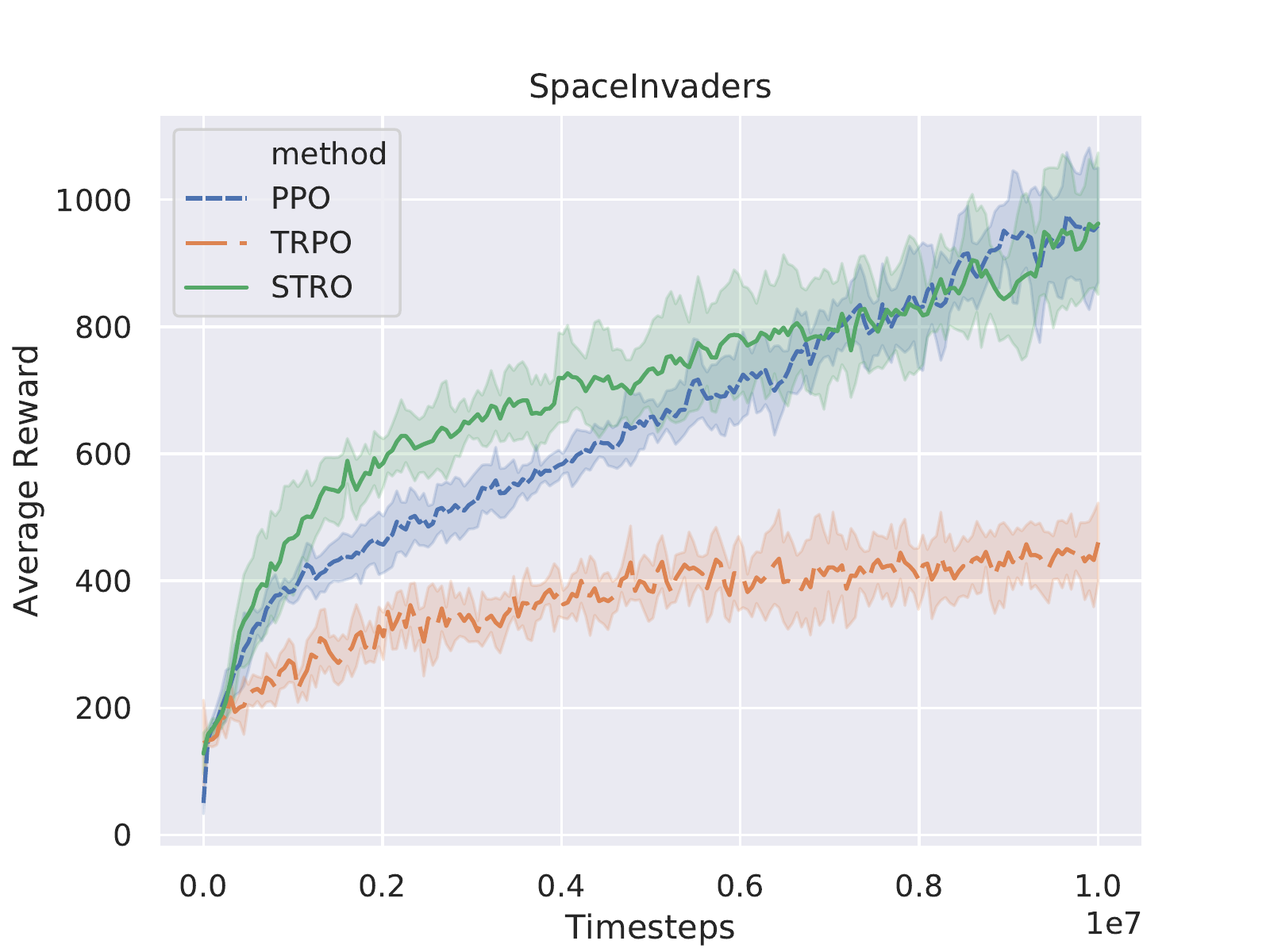}}
		\subfloat[]{\includegraphics[width=1.8in]{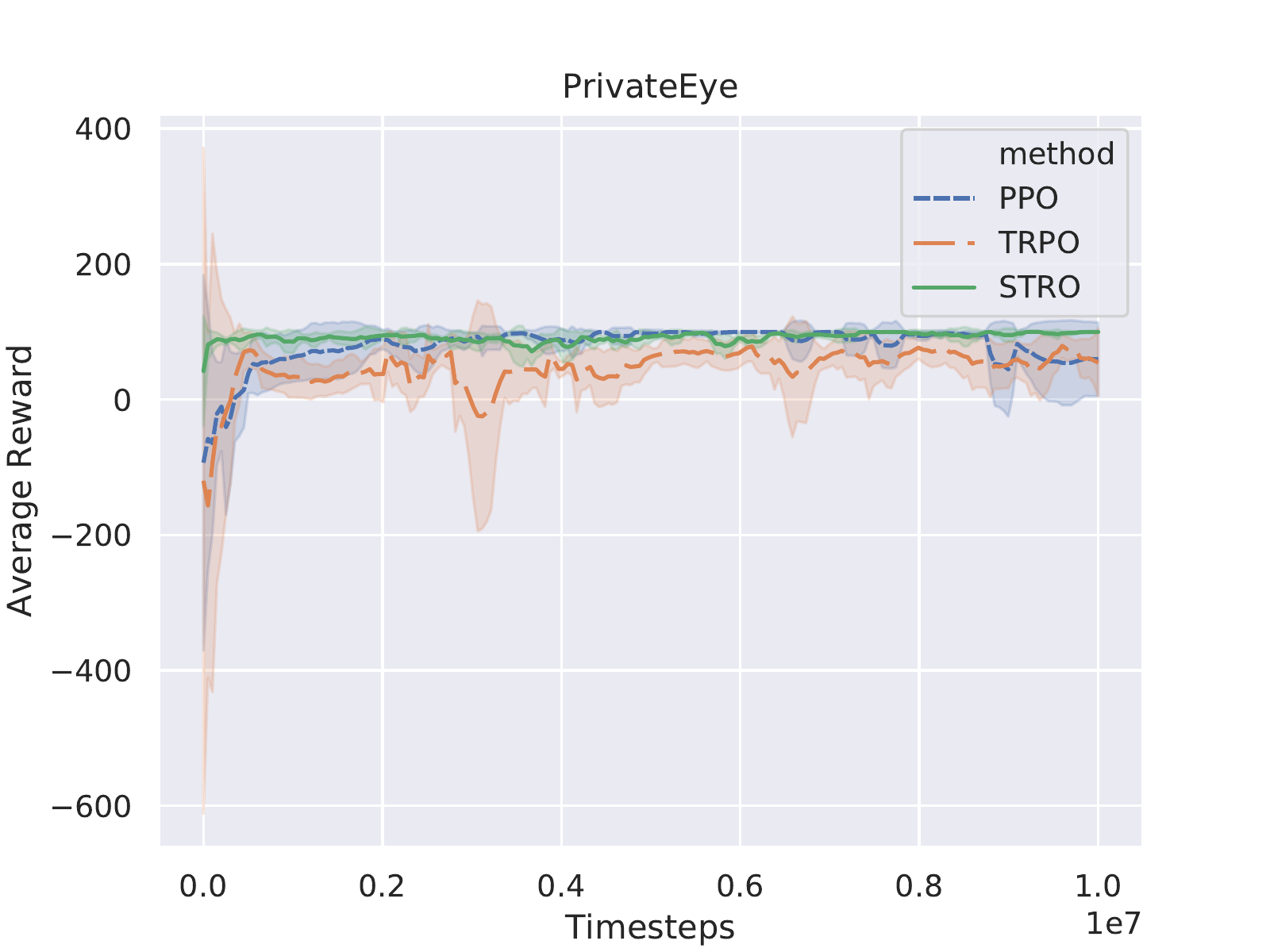}}
		\subfloat[]{\includegraphics[width=1.8in]{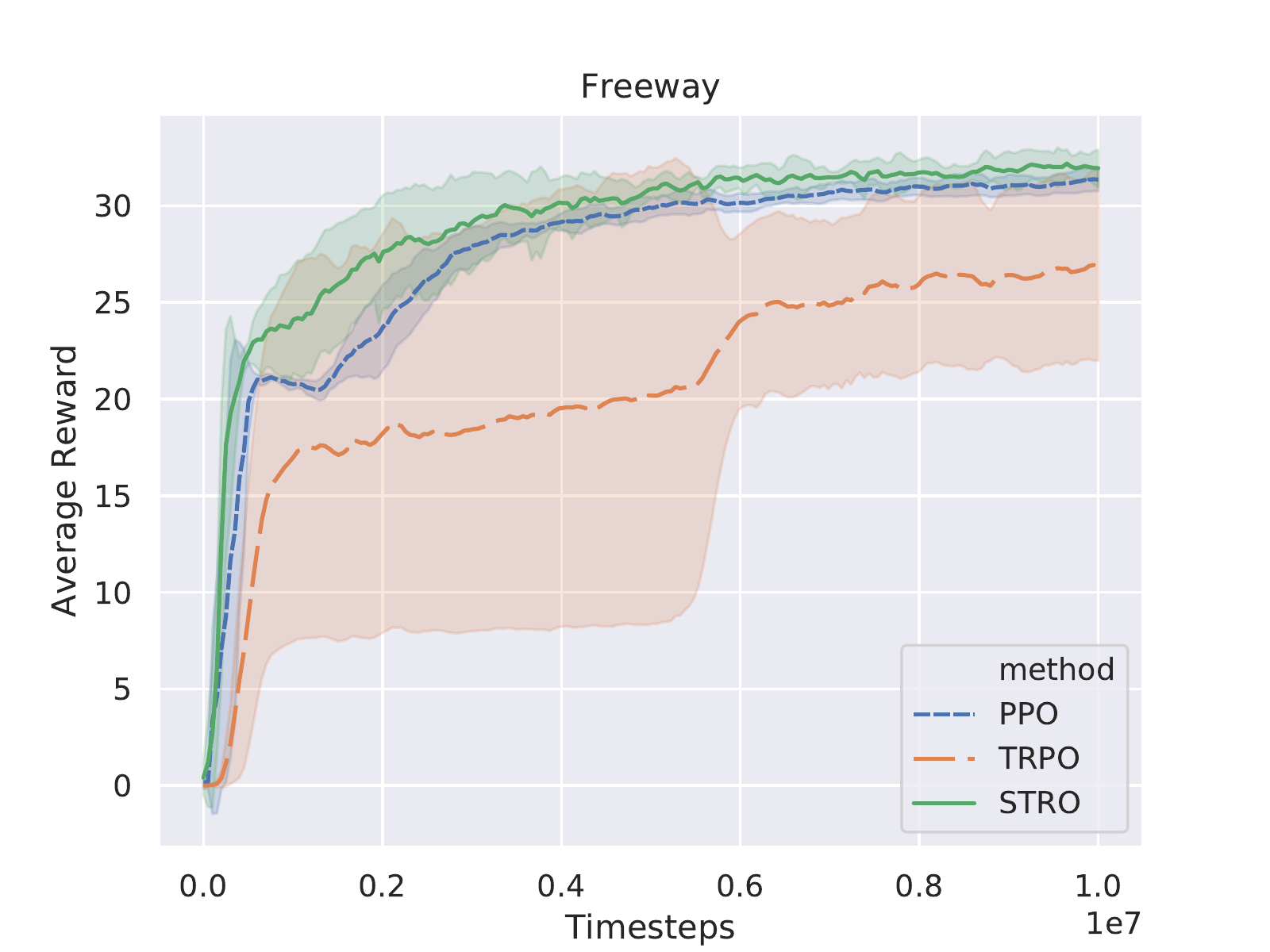}}

		\caption{Training curves on Atari games over 16 parallel programmings.}
		\label{fig:atari}
	\end{figure}

\begin{table}[h]
\caption{Max Average Reward $\pm$ standard deviation over 5 trials of 1e7 time steps.}
\begin{center}
\begin{tabular}{|c|c|c|c|c|}
\hline
Environment & PPO & TRPO  & STRO\\
\hline
BeamRider & 3113$\pm$375 & 760$\pm$30 & \textbf{3836}$\pm$903 \\
\hline
Bowling & 61$\pm$11 & 58$\pm$18 & \textbf{81}$\pm$18 \\
\hline
FishingDerby & 11$\pm$18 & -82$\pm$2 & \textbf{25}$\pm$0 \\
\hline
MsPacman & 2037$\pm$153 & 1538$\pm$159 & \textbf{2558}$\pm$442 \\
\hline
Pong & 19$\pm$0 & 3$\pm$7 & \textbf{20}$\pm$0 \\
\hline
Seaquest & 1100$\pm$317 & 692$\pm$92 & \textbf{1396}$\pm$398 \\
\hline
SpaceInvaders & 965$\pm$108 & 540$\pm$20 & \textbf{1010}$\pm$81 \\
\hline
PrivateEye & 100$\pm$0 &  88$\pm$16 & 100$\pm$0 \\
\hline
Freeway & 30$\pm$0 & 28$\pm$3 & \textbf{31}$\pm$0 \\
\hline
\end{tabular}
\end{center}
\label{maxcompare-atari}
\end{table}

\section{Conclusion}	
We propose a stochastic trust-region framework for policy optimization and a decoupled update for the Gaussian policy to avoid premature convergence. For the unparameterized policies, we prove the global convergence of the proposed algorithm under mild and feasible assumptions. Moreover, in the parameterized case, we show that the mean of the Gaussian policies converges to the stationary point where the covariance of the policies is assumed to converge. In robotic locomotion using a general-purpose policy network, we successfully learn better controllers than PPO and TRPO. Our method is able to surpass DDPG in most tested environments under given time constraints. Meanwhile, we show that our algorithm is comparable with the start-of-the-art deep reinforcement learning methods, such as TD3 and SAC. Our method is also suitable for discrete tasks such as playing some Atari games, and it can outperform PPO and TRPO in quite a few tasks.

\section*{Acknowledgement}
The computational results were obtained at GPUs 
 supported by the National Engineering Laboratory for Big Data Analysis and Applications and the High-performance Computing Platform of Peking University.

\bibliographystyle{siamplain}
\bibliography{interacttfssample}
\newpage

\end{document}